\documentclass[a4paper,11pt]{article}
\usepackage{amssymb,amsmath,amsthm}
\usepackage{verbatim}  
\usepackage{graphicx}
\usepackage{enumerate}
\usepackage[all]{xy}  
\usepackage{tikz}
\usepackage[pdftex]{hyperref}
\hypersetup{colorlinks, citecolor=black,
filecolor=black,linkcolor=black, urlcolor=black}
\usepackage{color}

\newtheorem{theorem}{Theorem} [section]
\newtheorem{lemma}[theorem]{Lemma}
\newtheorem{proposition}[theorem]{Proposition}
\newtheorem{corollary}[theorem]{Corollary}

\theoremstyle{definition}
\newtheorem{definition}[theorem]{Definition}

\newtheorem{example}[theorem]{Example}
\newtheorem{remark}[theorem]{Remark}

\newcommand{\N}{\mathbb{N}}

\newcommand{\emptyword}{\lambda}  

\newcommand{\power}{\wedge}   

\newcommand{\supp}{\mathsf{Supp}}

\newcommand{\sgf}{\theta}        

\newcommand{\neck}{\mathsf{N}} 

\newcommand{\Ne}{\mathsf{Lk}}        

\newcommand{\wordl}[1]{|#1|}  
\newcommand{\eltl}[1]{||#1||}  

\newcommand{\crseries}{\mathsf{\rho}}  
\newcommand{\dual}[1]{\overline{#1}}
\newcommand{\decomp}{\delta_{\ll}}

\newcommand\cycsl{\mathsf{CycSL}}   
\newcommand{\sphcpl}{\mathsf{CycSL}}   

\newcommand\geol{\mathsf{Geo}}   
   
\newcommand\geos{{\gamma}}    
\newcommand\wgeos{{geodesic growth series}}    
\newcommand\slex{<_{sl}}
\newcommand\slexeq{\leq_{sl}}

\newcommand\cycperm{\mathsf{CycPerm}}
\newcommand\prim{\mathsf{Prim}}
\newcommand\cycrep{\mathsf{CycRep}}

\newcommand\sphl{\mathsf{SL}}    
\newcommand\wsphl{{shortlex language}}    
\newcommand\sphs{{\sigma}}    
\newcommand\wsphs{{spherical growth series}}    

\newcommand\geocl{\mathsf{ConjGeo}}   
   
\newcommand\geocs{{\widetilde {\gamma}}}    
\newcommand\wgeocs{{geodesic conjugacy growth series}}  
  
\newcommand\geocpl{\mathsf{CycGeo}}   

\newcommand\sphcl{\mathsf{ConjSL}}  
\newcommand\wsphcl{{shortlex conjugacy language}}    
\newcommand\sphcs{{\widetilde {\sigma}}}    
\newcommand\wsphcs{{spherical conjugacy growth series}}    

\newcommand\slnf{{shortlex normal form}}
\newcommand\slcnf{{shortlex conjugacy normal form}}

\begin{document}

\title{Formal conjugacy growth in graph products II}
\author{L. Ciobanu, S. Hermiller, V. Mercier}
\maketitle

\begin{abstract}
In this paper we give an algorithm for computing the conjugacy growth series for a right-angled Artin group, based on a natural language of minimal length conjugacy representatives. In addition, we provide a further language of unique conjugacy geodesic representatives of the conjugacy classes for a graph product of groups. 
The conjugacy representatives and growth series here provide an alternate viewpoint, and are more amenable to computational experiments compared to those in our previous paper~\cite{chm}.

Examples of applications of this algorithm for right-angled Artin groups
are provided, as well as computations of conjugacy geodesic growth growth series with respect to the standard generating sets. 

\bigskip


\bigskip

\end{abstract}


\section{Introduction}\label{sec:intro}


The \textit{spherical standard growth function} of a group $G$ with respect to a finite inverse-closed generating set $X$ records the size of the sphere of size $n$ in the Cayley graph of $G$ with respect to $X$, for any $n\geq 0$, and the \textit{spherical conjugacy growth function} counts the number of conjugacy classes intersecting the sphere of radius $n$ but not the ball of radius $n-1$. 
Furthermore, the \textit{spherical standard growth series} and 
\textit{spherical conjugacy growth series} 
are those generating functions whose coefficients are the spherical growth function and spherical conjugacy growth function values, respectively.
For a group that has a regular language of shortlex normal forms, such as word hyperbolic groups and all shortlex automatic groups~\cite{echlpt},
the spherical standard growth series is a rational function that can be computed using functions in the KBMAG package in GAP~\cite{kbmag,GAP4}. 

However, even for nonabelian free groups with respect to the standard generating set, the spherical conjugacy growth series is not rational; this was shown by Rivin~\cite{R04, R10} in his work introducing the first conjugacy growth series computations, in which he showed the stronger result that the series for free groups are transcendental. 
Conjugacy growth series have also been studied in \cite{AC17, BdlH16, ceh20, CH14, CHHR14, chm, Crowe23, Evetts19, VM17}, where virtually abelian groups, Baumslag-Solitar groups, acylindrically hyperbolic groups, free and wreath products, and more, were explored. Except for virtually abelian groups, all of the spherical conjugacy growth series studied so far have been shown to be transcendental, and it has been conjectured that virtually abelian groups are the only ones with rational conjugacy growth series \cite[Conjecture 7.2]{ceh20}.

We continue work begun in~\cite{chm}, analyzing spherical conjugacy growth series and related languages for graph products of groups. 
In our previous paper~\cite{chm} we computed the conjugacy growth series of a graph product starting with the splitting of a graph product as an amalgamated product, finding minimal length conjugacy representatives among particular normal forms for amalgamated products, and computing the growth series of appropriate admissible transversals. 
However, that method does not explicitly produce a language of conjugacy representatives, and is somewhat technical. 
In this paper we study graph products of groups
from an alternate viewpoint, to obtain a
language of unique
conjugacy geodesic representatives of conjugacy classes 
in a graph product of groups that is built
from shortlex conjugacy languages for subgraph products.
This reduces the computation of the spherical
conjugacy growth series of the entire graph product to
computation of the spherical conjugacy growth series for
subgraph products over indecomposable  subsets of vertices,
in Theorem~\ref{thm:gpconjsldecomp}. 

Sections~\ref{sec:preliminaries},~\ref{subsec:gpbackground},
and~\ref{subsec:raaglanguages} contain complete definitions of the terminology and notation used below and throughout this paper.
Given a finite simple graph $\Gamma=(V,E)$ with a group $G_v$ generated by $X_v$
attached to each vertex $v\in V$, the associated \textit{graph product} $G_V$ is the
group generated by the vertex groups with the added relations that
elements of groups attached to adjacent vertices commute. We will often use the \emph{complement} of the graph $\Gamma$, which is the graph 
$\dual{\Gamma} = (V,\dual{E})$ with the same vertex set as $\Gamma$
and edge set that is the complement of $E$. 

In Section~\ref{sec:graphproduct}, we give a decomposition of the spherical conjugacy growth series of a graph product group $G_V$ in terms of 
subgraph products $G_U$
induced by subsets $U \subseteq V$. The subgraph $\dual{\Gamma_U}$ of $\dual{\Gamma}$ induced by $U$ consists of connected components induced by
subsets $U_1,...,U_n$ of $U$; we write this decomposition as 
$\decomp(U)=(U_1,...,U_m)$ (see Definition~\ref{def:decomp} for more details), 
and note that $G_U \cong G_{U_1} \times \cdots \times G_{U_n}$.
Given a total order $<_V$ on the generating set $X_V := \cup_{v \in V} X_v$
of $G_V$ and the corresponding shortlex order $<_{sl}$ on 
the set $X_V^*$ of words over $X_V$,
the \textit{shortlex conjugacy language} $\sphcl=\sphcl(G_V,X_V)$ 
is the set consisting of the shortlex least word representing
an element in each conjugacy class.
For any $U \subseteq V$, let $\sphcl^{U}$ denote the set of
words with \textit{support} equal to $U$;
that is, words
$w \in \sphcl$ such that $w$ contains only letters in
$X_U := \cup_{u \in U} X_u$ but $w$ is not contained in
$X_{U'}^*$ for any $U' \subsetneq U$.
In our earlier paper~\cite[Proposition~3.7]{chm} we showed
that  $\sphcl(G_V,X_V)^{U}=\sphcl(G_U,X_U)^{U}$
can be computed in the subgraph product $G_U$ generated by $X_U$.
The growth series $F_{\sphcl}$ of $\sphcl$ is
equal to 
$\sphcs_{(G_V,X_V)}$, but the growth of this
language is not readily computed;
in the following we provide another formula for $\sphcs_{(G_V,X_V)}$
in terms of the growth functions $F_{\sphcl^{U}}$.

\medskip

\noindent{\bf Theorem~\ref{thm:gpconjsldecomp}.}
\textit{
Let $G_V$ be a graph product group over a graph
with vertex set $V$. Let $X_V$ be
a union of inverse-closed generating sets for the
vertex groups,
let $<_V$ be a total order of $X_V$ 
that is compatible with
a total order $\ll$ of $V$, and let $<_{sl}$ be the corresponding
shortlex order on $X_V^*$.
Then
the language
\[
\{\emptyword\} \bigsqcup \left(\bigsqcup_{U \subseteq V,~\decomp(U)=(U_1,...,U_m)}
\sphcl^{U_1} \cdots \sphcl^{U_m}\right)
\] 
is a set of unique
conjugacy geodesic representatives for the conjugacy classes of $G_V$ over $X_V$,
and the spherical conjugacy growth series satisfies
\[
\sphcs_{(G_V,X_V)} = 1 + 
\sum_{\emptyset \ne U \subseteq V,~\decomp(U)=(U_1,...,U_m)}
F_{\sphcl^{U_1}} \cdots F_{\sphcl^{U_m}}.
\]
}

\medskip

For the remainder of the paper, we focus on the special case of
\textit{right-angled Artin groups}, or \textit{RAAGs}, 
in order to obtain a more computable formula for the spherical growth series
and calculate examples.
RAAGs are the graph products of infinite cyclic groups, and
the family of RAAGs has been widely studied and applied in many contexts.
Building on Theorem~\ref{thm:gpconjsldecomp}
and work of Crisp, Godelle, and Wiest~\cite{CGW},
we give another set of conjugacy geodesic
normal forms for the conjugacy classes of a RAAG
with respect to the \textit{Artin generating set} $X_V$,
consisting of a cyclic generator and its inverse for each
vertex group,
in Corollary~\ref{cor:conjrepseries}.  
The elements of this new set of conjugacy class representatives
are \textit{cyclic shortlex};
that is, they lie in the set
$\sphcpl=\sphcpl(G_V,X_V)$ of words $w$ satisfying
the property that for every cyclic permutation $w'$ of $w$,
the word $w'$
is shortlex least among the words representing the same element 
of $G_V$ as $w'$.

Unlike the languages $\sphcl$ and $\sphcl^{U}$, the languages
$\sphcpl$ and $\sphcpl^{U}$ 
(where $\sphcpl^{U}$ is the 
set of words in $\sphcpl$ with support equal to $U$)
are readily computed for RAAGs. 
The next main result of our paper gives a formula for 
the spherical growth series of a RAAG in terms of 
the strict growth series of the latter languages.

\medskip

\noindent{\bf Theorem~\ref{thm:formula}.}
\textit{
Let $G_V$ be a right-angled Artin group on a graph $\Gamma=(V,E)$
and let $X_V$ be the Artin generating set.
The spherical conjugacy growth series of $G_V$
with respect to $X_V$ satisfies
the formula
\begin{eqnarray}\label{eq:permformula}
\sphcs_{(G_V,X_V)} &=& 
1 +  \sum_{\emptyset \ne U \subseteq V,~\decomp(U)=(U_1,...,U_m)}
\crseries(F_{\sphcpl^{U_1}}) \cdots \crseries(F_{\sphcpl^{U_m}}) 
\end{eqnarray}
where
for any formal complex power series $f$ with
integer coefficients and constant term equal to 0, 
\begin{equation}\label{eq:transformP}
\crseries(f)(z) := 
\int_0^z \frac{\sum_{k\geq 1} \phi(k) f(t^k)}{t} \,dt.
\end{equation}
}

\medskip

The proof of Theorem~\ref{thm:formula}
combines the conjugacy class representatives in Corollary~\ref{cor:conjrepseries}
with a method from analytic combinatorics for counting words up to cyclic permutation;
since this tool is not standard in group theory, we give a full account of the background for the operator $\crseries$ 
in Section~\ref{subsec:cycrep}.
In particular, 
when $F_L$ is the growth series of a 
set $L$ of nonempty words over $X$ that is closed under cyclic permutation
and taking powers, $\crseries(F_L)$ is the growth function of
a subset of $L$ containing exactly one representative of each
cyclic permutation class from $L$~(see Corollary~\ref{cor:decycle}).


In Part~I of this series of papers~\cite{chm}, we provided a
different formula for the spherical conjugacy growth series of 
a RAAG, and more generally for a graph product $G_V$, 
based on the spherical growth series
$\sphs_U$ and spherical conjugacy growth series $\sphcs_U$ of 
the subgraph products $G_U = \langle X_U\rangle$ 
induced by subsets $U \subset V$. 
In~\cite[Theorem~A]{chm} the formula
is obtained
from splitting the graph product as an amalgamated product; 
given any vertex $v \in V$, 
the spherical growth
series of the graph product $G_V$ 
is given by
\begin{equation}\label{eq:necklaceformula}  
\sphcs_V=\, \sphcs_{V\setminus\{v\}}+\sphcs_{\Ne(v)}(\sphcs_{\{v\}}-1) +
   \sum_{S\subseteq \Ne(v)} \sphcs^{\mathcal{M}}_S~ 
  \neck\left(\left(\frac{\sphs_{\Ne(S)\setminus\{v\}}}
    {\sphs_{\Ne(v)\cap\Ne(S)}}-1\right)(\sphs_{\{v\}}-1)\right),
\end{equation}
where 
$\Ne(v)$ is the set of vertices adjacent to $v$
in the defining graph $\Gamma=(V,E)$ of the graph product,
$\sphcs^{\mathcal{M}}_{S} := \sum_{S'\subseteq S} (-1)^{|S|-|S'|} \sphcs_{S'}$,
 and for any formal complex power series $f(z)$,
\begin{equation}\label{eq:transformN}
\neck(f)(z):=\sum_{k=1}^\infty\sum_{l=1}^\infty \frac{\phi(k)}{kl}\big(f(z^k)\big)^l
\end{equation}
in which $\phi$ is the Euler totient function.
In contrast to Theorem~\ref{thm:formula}, the formula in
Equation~(\ref{eq:necklaceformula}) is recursive,
requiring iteration to compute the spherical conjugacy
growth series of subsequent subgraph products.
For the operator $\neck$,
when $F_L$ is the growth series of a 
set $L$ of words over $X$,
$\neck(F_L)$
is the growth function of
$\bigcup_{n=1}^\infty (L^n/C_n)$,
where $L^n$ denotes a Cartesian product of $n$ copies of $L$
and the cyclic group of order $n$, $C_n$, acts by permutation.
In  words, with the operator $\neck$
in Equation~(\ref{eq:necklaceformula})
we take Cartesian products of a language and then 
quotient by appropriate cyclic group actions,
whereas with the operator $\crseries$
in Equation~(\ref{eq:permformula})
we take the quotient on the language itself.


As a first illustration of the potential for applications 
of Theorem~\ref{thm:formula},
our next corollary provides a formula for 
computing the spherical growth series
of a free product of a RAAG with a free group.

\medskip

\noindent{\bf Corollary~\ref{cor:freeprodwithfree}.}
\textit{
Let $G_{W}$ be a right-angled Artin group on a simple graph $\Gamma_W=(W,E)$,
and let $G_W \ast F_m$ be the free product of $G_{W}$ with
a free group of rank $m \ge 1$ with free basis $W'$. 
Viewing $G_W \ast F_m$ as a RAAG 
with defining graph $\Gamma_V := (V,E)$ where $V := W \sqcup W'$,
let $X_{W}$ and $X_V$ be the Artin generating sets
of $G_W$ and $G_W \ast F_m$, respectively.
Then 
\[
\sphcs_{(G_W \ast F_m,X_V)} = 
\sphcs_{(G_W,X_W)} +
\crseries\left(F_{\sphcpl(G_W \ast F_m,X_V)} - F_{\sphcpl(G_W,X_W)} \right).
\]
}

\medskip

In Section~\ref{sec:examples} we discuss methods for computing
the languages $\sphcpl^U$ and their strict growth series
$F_{\sphcpl^U}$ for right-angled Artin groups
using the KBMAG~\cite{kbmag} package in the computational
algebra software GAP~\cite{GAP4}.
Because the computations are quite similar,
we also consider computations of the \textit{conjugacy geodesic growth series}
$\geocs_{(G_V,X_V)}$ for these groups, which are
the formal complex power series whose $n$-th coefficient
counts the number of conjugacy geodesic words of length $n$.
We accompany this discussion with specific 
examples of GAP code
provided in the Appendix (Section~\ref{sec:appendix}).

For the remainder of Section~\ref{sec:examples},
we apply these computational tools along with 
Theorem~\ref{thm:formula} to compute spherical
conjugacy and conjugacy geodesic growth series
for some families of RAAGs.
In Example~\ref{ex:freeabelian}
we compute the conjugacy geodesic growth series for
free abelian groups of small rank, and extrapolate that 
information to obtain an
inductive proof of the following for all 
finitely generated free abelian groups:

\medskip

\noindent{\bf Theorem~\ref{thm:freeabgeo}.}
\textit{
The geodesic and conjugacy geodesic growth series
for $\mathbb{Z}^n$ with respect to the standard generating set $X$
satisfy
\[ 
\geocs_{(\mathbb{Z}^n,X)} = \geos_{(\mathbb{Z}^n,X)} = 
1 + \sum_{j=1}^n (-1)^{n-j} (2^j) \binom{n}{j} \frac{jz}{1-jz}
\]
for all $n \ge 0$.
}

Theorem~\ref{thm:freeabgeo} can be obtained directly by geometric arguments using the inclusion-exclusion principle on the numbers of geodesic paths of a given length in the $2^n$ orthants of $\mathbb{Z}^n$, but instead of such arguments, the emphasis in Section~\ref{sec:examples} is on how to use our experiments to find the above formula.


\section{Preliminaries}\label{sec:preliminaries}



\subsection{Notation and terminology}\label{subsec:notation}


We use standard notation from formal language theory:
For any finite set $X$,
we denote by $X^*$ the set of all words over $X$, and call a subset of
$X^*$ a {\em language} over $X$.
We write $\emptyword$ for the empty word, 
and denote by $X^+$ the set of all non-empty words over $X$ 
(so $X^*=X^+ \cup \{ \emptyword \}$).
For each word $w \in X^*$, let $l(w)=l_X(w)=\wordl{w}$ denote its length over $X$.

All groups in this paper are finitely generated, and all generating sets finite and inverse-closed. 
For a group $G$ with inverse-closed generating set $X$, 
let $\pi:X^* \rightarrow G$ be the natural projection onto $G$, and let
$=$ denote equality between words and $=_G$ equality between group elements
(so $w=_G v$ means $\pi(w)=\pi(v)$).
For $g \in G$, the {\em length}
of $g$, denoted $\eltl{g}\;(=\eltl{g}_X)$, is
the length of a shortest representative word for $g$ over $X$.

Let $\sim$, or $\sim_G$, denote the equivalence relation on $G$ given
by conjugacy, and $G/\sim$ its set of equivalence classes.
Let $[g]_{\sim}$ denote the conjugacy class of $g \in G$ and $\eltl{g}_{\sim}$ denote
its {\em length up to conjugacy}, that is, \[\eltl{g}_{\sim}:=\min\{\eltl{h} \mid h \in [g]_{\sim}\}.\]

A word $w \in X^*$ is a {\em geodesic} 
if $l(w)=\eltl{\pi(w)}$, and $w$ is a
 {\em conjugacy geodesic} if $l(w)=\eltl{\pi(w)}_{\sim}$.
The {\em geodesic language}
 and {\em conjugacy geodesic language} for $G$ over $X$ are defined as 
\begin{eqnarray*} \label{eq:geo}
\geol = \geol(G,X)&:=&\{w \in X^* \mid l(w)=\eltl{\pi(w)}\},  \\
\geocl = \geocl(G,X)&:=&\{w \in X^* \mid l(w)=\eltl{\pi(w)}_{\sim}\}.
\end{eqnarray*}

Fix a total order of $X$, and 
let $\slexeq$ be 
the induced shortlex order of $X^*$ (for which
$u \slex w$ if either $l(u)<l(w)$, or $l(u)=l(w)$ but
$u$ precedes $w$ lexicographically).
For each $g \in G$, the {\em \slnf} of $g$ is the unique word
$y_g \in X^*$ with $\pi(y_g)=g$ such that
$y_g \slexeq w$  for all $w \in X^*$
with $\pi(w)=g$.
For each conjugacy class $c \in G/\sim$, the 
{\em \slcnf} of $c$ 
is the shortlex least word $z_c$ over $X$ representing an 
element of $c$; that is, 
$\pi(z_c)\in c$, and $z_c \slexeq w$ for all 
$w \in X^*$ with $\pi(w)\in c$.
The {\em \wsphl}
 and {\em \wsphcl} for $G$ over $X$ are defined as 
\begin{eqnarray*} \label{rep_defs1}
\sphl = \sphl(G,X)&:=&\{y_g \mid g \in G\},  \\
\sphcl = \sphcl(G,X)&:=&\{z_c \mid c \in G/\sim\}.
\end{eqnarray*}

A {\em cyclic permutation} of a word
$w \in X^*$ is a word of the form $vu$ such that
$u,v \in X^*$ and $w=uv$.
A word $w$ is
{\em cyclically geodesic} over $X$
if every cyclic permutation of $w$ lies in $\geol(G,X)$.
Similarly, given a total order on $X$,
a word $w$ over $X$ is 
{\em cyclically shortlex} if every cyclic permutation
of $w$ lies in $\sphl(G,X)$.
The {\em cyclically geodesic language}
 and {\em cyclically shortlex language} for $G$ over $X$ are defined as
\begin{eqnarray*}
\geocpl = \geocpl(G,X) &:=& \{ \text{cyclically geodesic words for }G 
  \text{ over }X\},  \\
\sphcpl = \sphcpl(G,X) &:=& \{ \text{cyclically shortlex words for }G 
  \text{ over }X\}.
\end{eqnarray*}

Any language $L$ over $X$ gives rise to a
{\em strict growth function} $\sgf_L:\N \cup \{0\} \to \N \cup \{0\}$, 
defined by  $\sgf_L(n) := |\{w \in L \mid l(w) = n\}|$;
an associated generating function, called
the {\em strict growth series}, is given
by $F_L(z) := \sum_{n=0}^\infty \sgf_L(n)z^n$.

For the languages $\sphl$ and $\sphcl$ above, the coefficient
$\sgf_{\sphl}(n)$ is the number of elements of $G$  of length $n$,
  and $\sgf_{\sphcl}(n)$ is the number of conjugacy classes of $G$ whose shortest
elements have length $n$.
As in~\cite{chm}, we
refer to the strict growth series of $\sphl$ below as the \emph{standard} or {\em \wsphs} 
\[ \sphs(z)=\sphs_{(G,X)}(z) := F_{\sphl(G,X)}(z)
  = \sum_{n=0}^\infty \sgf_{\sphl(G,X)}(n)z^n \]
and the strict growth series 
\[ \sphcs(z)=\sphcs_{(G,X)}(z) := F_{\sphcl(G,X)}(z)
  = \sum_{n=0}^\infty \sgf_{\sphcl(G,X)}(n)z^n, \]
of $\sphcl$ as the {\em \wsphcs}.

Note that the growth series in the paper will be often denoted as 
$\sphs$ and $\sphcs$ instead of $\sphs(z)$ or $\sphcs(z)$ 

Similarly, we refer to the strict growth series of $\geol$ 
below as the {\em \wgeos} 
\[ \geos(z)=\geos_{(G,X)}(z) := F_{\geol(G,X)}(z)
  = \sum_{n=0}^\infty \sgf_{\geol(G,X)}(n)z^n \]
and the strict growth series 
\[ \geocs(z)=\geocs_{(G,X)}(z) := F_{\geocl(G,X)}(z)
  = \sum_{n=0}^\infty \sgf_{\geocl(G,X)}(n)z^n, \]
of $\geocl$ as the {\em \wgeocs}.



\subsection{Cyclic representatives of languages}\label{subsec:cycrep}


In this section we provide a tool for computing the conjugacy growth series of a right-angled Artin group that will be applied in Section~\ref{sec:conjrepraag}.
This language theoretic analysis appears in several references (\cite{flajoletsoria, flajoletsedgewick}) in analytic combinatorics, but since it is not standard in group theory, we provide here the computations for the formulas.

Let $X$ be a finite set. If $L$ is a language over $X$, then
define
\[ \cycperm(L) := \{vu : uv \in L\},\]
the set of all cyclic permutations of words in $L$.
As an example, we note that for any group $G$ with
finite inverse-closed generating set $X$, a cyclic permutation
of a word $w$ over $X$ represents a conjugate of $w$ in $G$, and so
the language $\geocl(G,X)$ of conjugacy geodesics is closed under
cyclic permutation; hence $\cycperm(\geocl(G,X)) = \geocl(G,X)$.
For any language $L \subseteq X^*$ satisfying $\cycperm(L)=L$, 
we further define an equivalence relation on $L$ by 
$uv \sim vu$ for all $uv \in L$; 
equivalence classes are called \emph{cycles},
or \emph{cyclic permutation classes}, of $L$.

\begin{definition}\label{def:cycrep}
Let $X$ be a finite set with a total order and
let $L$ be a language over $X$ satisfying $\cycperm(L)=L$.
For any word $w \in L$ let $w_p$ be the word that is lexicographically 
least among all cyclic permutations of $w$. 
The language of \emph{cyclic representatives} of $L$ is
$$
\cycrep(L):=\{w_p \mid w \in L\}.
$$ 
\end{definition}

Since all of the words in a cycle have the same length,
the length function on words in $L$ can be extended to the set of cycles of $L$.
Note that there is a length-preserving bijection between $\cycrep(L)$ and the set of cycles of $L$.

The goal of this section is to give a formula for the growth
series of $\cycrep(L)$ in terms of the growth series for $L$.
We first need further notation and lemma.

For each $n \in \N$ let $L^{\power n}=\{u^n \mid u \in L\}$ 
denote the set of all $n$th powers of words in $L$.
The language of \emph{primitive words} in $L$ is defined to be 
$$
\prim(L):=\{w \in L \mid \nexists~k>1, v \in L \textrm{ such that } v^k=w \}.
$$
(For example, if L is the set $\prim(A^*)$ of all primitive words over $A$, then
$\cycrep(L)$ is the set of Lyndon words over $A$.)  

\begin{lemma} \label{lem:CRseriesidentities}
Let $X$ be a finite set, let
$L$ be a language over $X$, and let $n\in \N$.  
Then the following hold.
\begin{enumerate}
\item\label{lem:power} $F_{L^{\power n}}(z)=F_L(z^n)$.
\item\label{lem:primsum} If $L$ does not contain the empty word, then
$F_L(z)=\sum_{k \ge 1} F_{\prim(L)}(z^k)$.
\end{enumerate}
\end{lemma}

\begin{proof}
For completeness we include details of the proof of part~\ref{lem:primsum}.
Since $\emptyword \notin L$, 
then every word in $L$ can be written uniquely as a 
power of a primitive word; that is, 
$$L=\bigsqcup_{k\geq 1}(\prim(L))^{\power k},$$
and so $F_L(z)=\sum_{k \ge 1} F_{\prim(L)^{\power k}}(z)=
\sum_{k \ge 1} F_{\prim(L)}(z^k)$, using part~\ref{lem:power} of 
Lemma~\ref{lem:CRseriesidentities}.
\end{proof}


\begin{proposition}\label{prop:decycle}
Let $L \subseteq X^*$ be a language satisfying 
$L^{\power k} \subseteq L$ for all $k \geq 1$,
$\emptyword \notin L$, and $\cycperm(L)=L.$
Then the generating function of $\cycrep(L)$ is 
$$F_{\cycrep(L)}(z)= \int_0^z \frac{\sum_{k\geq 1} \phi(k) F_L(t^k)}{t} \,dt,$$
where $\phi$ is the Euler totient function.
\end{proposition}

\begin{proof}
Notice that
for any word $w \in (\prim(L))^{\power k}$, the equivalence class (cycle)
of $w$ is also contained in $(\prim(L))^{\power k}$, and so
$\cycrep((\prim(L))^{\power k})=\cycrep(\prim(L))^{\power k}$.
Moreover, $\cycrep(L)=\bigsqcup_{k\geq 1}\cycrep((\prim(L))^{\power k});$ 
then Lemma~\ref{lem:CRseriesidentities} part \ref{lem:power} yields
\begin{equation}\label{eq:cycrep}
F_{\cycrep(L)}(z)=
\sum_{k \ge 1} F_{\cycrep(\prim(L))}(z^k).
\end{equation}

The number of cycles of length $n>0$ with  representatives in $\prim(L)$ is 
$\frac{1}{n}\sgf_{\prim(L)}(n)$, and
similarly, the number of distinct cycles of length $nk$ with representatives 
in $(\prim(L))^{\power k}$ is also
$\frac{1}{n}\sgf_{(\prim(L))^{\power k}}(nk)=\frac{1}{n}\sgf_{\prim(L)}(n)$.
Thus for a fixed $k$ a formal power series manipulation gives
\begin{eqnarray*}
F_{\cycrep(\prim(L))}(z^k)
&=& \sum_{n=1}^{\infty} \sgf_{\cycrep(\prim(L))}(n)(z^k)^n
= \sum_{n=1}^{\infty} \frac{k}{kn}\sgf_{\prim(L)}(n)z^{kn} \\
&=& \int_0^z k\sum_{n=1}^{\infty}\sgf_{\prim(L)}(n)t^{kn-1} \,dt
= \int_0^z \frac{ k F_{\prim(L)}(t^k)}{t} \,dt,
\end{eqnarray*}
and using Equation~(\ref{eq:cycrep})
$$
F_{\cycrep(L)}(z)= \int_0^z \frac{\sum_{k\geq 1} k F_{\prim(L)}(t^k)}{t} \,dt.$$

It remains to show that 
$$
\sum_{k\geq 1} k F_{\prim(L)}(t^k)=\sum_{k\geq 1} \phi(k) F_L(t^k).
$$ 
We use the number theoretic identity $k=\sum_{d/k}\phi(d)$ and write 
\begin{eqnarray*}
\sum_{k\geq 1} k F_{\prim(L)}(t^k)&=&
\sum_{k\geq 1} \sum_{d/k}\phi(d) F_{\prim(L)}(t^k) \\
&=&\sum_{k\geq 1} \sum_{d/k, m=\frac{k}{d}}\phi(d) F_{\prim(L)}(t^{md})=
\sum_{d\geq 1} \phi(d) \sum_{m\geq 1}F_{\prim(L)}(t^{md})\\
&=& \sum_{d\geq 1} \phi(d) F_L(t^d),
\end{eqnarray*}
using in the last equality the fact that 
$F_L(z) = \sum_{k \geq 1} F_{\prim(L)}(z^k)$ from
Lemma~\ref{lem:CRseriesidentities} part~\ref{lem:primsum}.
\end{proof}

Proposition~\ref{prop:decycle} leads us to define
the following operator on formal complex power series.

\begin{definition}\label{def:cycrepfunction}
For any formal complex power series $f$ with
integer coefficients and constant term equal to 0, let
\begin{equation}\label{rho_series}
\crseries(f)(z) := 
\int_0^z \frac{\sum_{k\geq 1} \phi(k) f(t^k)}{t} \,dt.
\end{equation}
\end{definition}

The following is an immediate corollary of Proposition~\ref{prop:decycle}.

\begin{corollary}\label{cor:decycle}
If $X$ is a finite set and 
$L \subseteq X^*$ is a language satisfying 
$L^{\power k} \subseteq L$ for all $k \geq 1$,
$\emptyword \notin L$, and $\cycperm(L)=L$,
then
$$F_{\cycrep(L)} = \crseries(F_L).$$
\end{corollary}

\begin{remark}\label{rmk:rhoisadditive}
For our computations in
Sections~\ref{sec:conjrepraag} and~\ref{sec:examples}, it will
be useful to note that this operator $\crseries$ on power series
satisfies the property that $\crseries(f+g) = \crseries(f)+\crseries(g)$
for any formal complex power series $f$ and $g$.
\end{remark}

\begin{example}\label{rmk:infinitecyclic}
The growth series $F_L = \frac{2z}{1-z}$
of the language 
$L= \{a^i \mid i \in \mathbb{Z} \setminus \{0\}\}$ 
of nonempty reduced words
over $X = \{a,a^{-1}\}$ in
the infinite cyclic group $G$ on $a$ satisfies
$\crseries\left(\frac{2z}{1-z}\right) = \frac{2z}{1-z}$,
since $L$ satisfies the hypotheses of Corollary~\ref{cor:decycle} and
$\cycrep(L) = L$.
\end{example}

We give applications of Corollary~\ref{cor:decycle} in 
Section~\ref{sec:conjrepraag} and show how 
it can be used to compute the conjugacy growth series in right-angled Artin groups. 
In particular, by letting $L$ be the language of nonempty 
cyclically reduced words in a 
free group and $F_L$ its corresponding generating function, we can use Corollary \ref{cor:decycle} and Equation~(\ref{rho_series})
to obtain Rivin's~\cite{R10} formula 
(Equation~(\ref{eq:freegroup}))
for computing the conjugacy growth series of free groups. 


\begin{theorem}\cite[Theorem 14.6, Corollaries 14.1, 14.5]{R10}\label{thm:rivincorrected}
The spherical conjugacy growth series for the free group on $k$ generators $F_k$,
with respect to the inverse closure $X$ of a free basis, satisfies
\begin{equation}\label{eq:freegroup}
\sphcs_{(F_k,X)}(z)= 1 + \crseries(F_L)=
1+\int_0^z \frac{\sum_{k\geq 1} \phi(k) F_L(t^k)}{t} \,dt
\end{equation}
where $F_L(z)$ is the strict growth series for the 
language $L$ of nonempty
cyclically reduced words on the standard generators 
in the free group $F_k$.
Moreover,  
\begin{equation*}
F_L(z)= \frac{1}{1-(2k-1)z}+\frac{1}{1-z}+\frac{2(k-1)}{1-z^2}-2k.
\end{equation*}
\end{theorem}

\begin{remark}\label{rmk:correctingrivin}
The formula in Theorem~\ref{thm:rivincorrected} is a consequence
of Rivin's proofs of Theorem~14.6 and Corollaries~14.1 and~14.5
in~\cite{R10}, and is a slight correction of the formula stated in
Theorem~14.6 of that paper. In particular, in the statements
of Corollary~14.5 and Theorem~14.6 of~\cite{R10},
the constant term 1 (for the empty word) was incorrectly added in the 
formula for the generating function 
of the sequence
$r\Theta_{\sphcl - \{\emptyword\}}(r)$.
\end{remark}


\subsection{Regular languages}\label{subsec:reg}


We  refer to \cite{hu} for details on regular languages
and finite state automata in this section.

For subsets $A,B$ of $X^*$, the language $AB$ is the set of concatenations
$wu$ with $w \in A$ and $u \in B$. 
Similarly for each $n \in \N \cup \{0\}$ the set $A^n$
is defined by
$A^0 := \{\emptyword\}$, $A^1 := A$, and for each
$n \ge 2$, $A^n := A^{n-1}A$ is the set of
concatenations of $n$ words from $A$.
Let $A^* := \cup_{n=0}^\infty A^n$ and $A^+ = \cup _{n=1}^\infty A^n$.

A language over $X$ is {\em regular} if it can 
be built out of finite subsets of $X$ using the operations of union,
intersection, complementation,
concatenation, and $*$; 
such an expression for a
language is called a {\em regular expression}. 

Regular languages are exactly the languages
accepted by finite state automata.
A \emph{finite state automaton}, or \emph{FSA},
is a 5-tuple $M = (X,Q,q_0,P,\delta)$
where $X$ is a finite set called the \emph{alphabet},
$Q$ is a finite set known as the \emph{set of states},
$q_0 \in Q$ is the \emph{initial state},
$P \subseteq Q$ is the \emph{set of accept states},
and $\delta:Q \times X \rightarrow Q$ is known as
the \emph{transition function}.
The transition function can be extended 
to $\widehat\delta:Q \times X^* \rightarrow Q$
recursively, by defining
$\widehat\delta(q,wx) := \delta(\widehat\delta(q,w),x)$
for all $q \in Q$, $w \in X^*$, and $x \in X$.
The \emph{language accepted by the FSA} $M$
is the set of all words $w \in X^*$
such that $\widehat\delta(q_0,w) \in P$.

For any regular language $L$, the strict
growth series $F_L(z)$ is a rational function.
This can be effectively computed, and in
particular the 
\texttt{GrowthFSA} command of the
KBMAG package~\cite{kbmag} in
the GAP computational algebra software~\cite{GAP4}
implements the computation of this
rational function for an FSA input.


\section{Graph products and a decomposition of
$\sphcs$
}\label{sec:graphproduct}


In this section we obtain a
language of unique
conjugacy geodesic representatives of conjugacy classes 
in a graph product of groups that is built
from shortlex conjugacy languages for subgraph products.
This reduces the computation of the spherical
conjugacy growth series of the entire graph product to
computation of the spherical conjugacy growth series for
subgraph products over indecomposable  subsets of vertices,
in Theorem~\ref{thm:gpconjsldecomp}. 
In Section~\ref{sec:conjrepraag} we use these results
in the case of graph products of infinite cyclic 
groups (that is, RAAGs) to produce yet another language
of conjugacy geodesic representatives that is more
amenable to computation.

We begin in Section~3.1 with terminology, notation, and 
known results on graph products of groups.


\subsection{Background on graph products and languages}\label{subsec:gpbackground}


Let $\Gamma=(V,E)$ be a finite simple graph;
that is, an undirected graph without loops or multiple edges
with vertex set $V$ and edge set 
$E$ consisting of pairs of elements of $V$.
Also throughout Section~\ref{sec:graphproduct},
let $\ll$ be a fixed total order of $V$.

For each vertex $v$ of $\Gamma$, let 
$G_v$ be a nontrivial group with an inverse-closed generating set $X_v$.
The {\em graph product of the groups $G_v$
with respect to $\Gamma$}, denoted $G_{\Gamma}$
or $G_V$, is the quotient of 
the free product of the groups $G_v$ by the normal
closure of the set of relators $[g_v,g_w]$ for all 
$g_v \in G_v$ and $g_w \in G_w$
such that $\{v,w\} \in E$. The group $G_V$
is generated by the (inverse-closed) set $X_V = \cup_{v \in V} X_v$.

For any subset $U \subseteq V$, the graph \emph{induced} by
$U$ is the graph $\Gamma_U = (U,E_U)$ with edges
$E_U = \{\{a,b\} \in E \mid a,b \in U\}$.
The {\em subgraph product} of $G_V$ associated to $U$ is the subgroup
$G_{U} := \langle \{G_v \mid v \in U\} \rangle$ of $G_V$.
By~\cite[Proposition~3.31]{G90}, $G_{U}$ is isomorphic to the graph product of
the groups $G_v$ ($v \in U$) on the induced subgraph $\Gamma_U$.
Note that $G_\emptyset$ is the trivial group.

The {\em support} of a word $w \in X_V^*$
is the set 
$$
\supp(w) = \{ v \in V \mid \text{ a letter in }X_v \text{ appears in }w\}.
$$

%

We will often work with the complement graph rather
than with $\Gamma$ itself.
The \emph{complement} of the graph $\Gamma$ is the graph 
$\dual{\Gamma} = (V,\dual{E})$ with the same vertex set as $\Gamma$
and edge set that is the complement of $E$; that is,
$$
\dual{E} := \{\{u,v\} \mid u,v \in V \text{ and } \{u,v\} \notin E\}.
$$
Note that if $U$ and $U'$ are subsets of $V$ that
induce distinct connected components of $\dual{\Gamma}$,
then for any $g \in G_U$ and $g' \in G_{U'}$, the elements
$g$ and $g'$ commute in the graph product $G_V$.
This yields the following.

\begin{remark}\label{rmk:dirprod}
If $V_1,...,V_m \subset V$ induce the connected 
components of $\dual{\Gamma_V}$, then the graph product 
$G_V  \cong G_{V_1} \times \cdots \times G_{V_m}$
is the direct product of the subgraph products $G_{V_i}$.
\end{remark}





%

A total order $<_V$ of the generating set $X_V = \cup_{v \in V} X_v$ 
of a graph product group $G_V$ on a graph $\Gamma=(V,E)$
is called \emph{compatible with}
a total order $\ll$ of the vertex set $V$
of $\Gamma$ if 
 for each vertex $v \in V$ there is a total order
$<_v$ of the set $X_v$ such that for all $a,b \in X_V$
we have $a <_V b$ if and only if either 
$\supp(a) =\{u\}$ and $\supp(b) = \{u'\}$ with $u \ll u'$,
or $\supp(a) = \supp(b)=\{u\}$ and $a <_{u} b$.

\begin{proposition}\cite[Cor.~3.3 and Prop.3.7]{chm}\label{prop:slsubset}
Let $G_V$ be a graph product group with generating set $X_V$, and
let $U$ be any subset of $V$.  
Let $\slex$ be a shortlex order on $X_V^*$ induced by an order on $X_V$
compatible with a total order $\ll$ on $V$, and let the shortlex
order on $X_{U}^*$ be the restriction of the order $\slex$.
Then
\begin{eqnarray*}
\sphl(G_V,X_V) \cap X_{U}^* &=& \sphl(G_{U},X_{U}), \\
\sphcpl(G_V,X_V) \cap X_{U}^* &=& \sphcpl(G_{U},X_{U})\\
\sphcl(G_V,X_V) \cap X_{U}^* &=& \sphcl(G_{U},X_{U}),\\
\geol(G_V,X_V) \cap X_{U}^* &=& \geol(G_{U},X_{U}),\\
\geocpl(G_V,X_V) \cap X_{U}^* &=& \geocpl(G_{U},X_{U}), \text{ and}\\
\geocl(G_V,X_V) \cap X_{U}^* &=& \geocl(G_{U},X_{U}).\\
\end{eqnarray*}
\end{proposition}

The {\em shuffle} operation on a word $w \in X_V^*$ is an application 
of a commuting relation from the graph product construction;
that is, if we can write $w = yabz$ with $y,z \in X_V^*$,
$a \in X_v^+$, and $b \in X_{v'}^+$ satisfying $\{v,v'\}$
is an edge of $\Gamma$, then there is a shuffle operation
from $w$ to the word $ybaz$.
We write $w \stackrel{sh}{\longrightarrow} w'$ if there is
a finite sequence of shuffle operations from $w$ to $w'$.


\subsection{A decomposition of $\sphcs(G_V,X_V)$ for graph products}\label{subsec:decompconjsl}


Let $<_V$ be a total order of the generating set $X_V = \cup_{v \in V} X_v$ 
of a graph product group $G_V$ on a graph $\Gamma=(V,E)$
that is compatible with
a total order $\ll$ of the vertex set $V$
of $\Gamma$.

Given any language $L \subset X_V^*$ and 
subset $U \subseteq V$, let
\[
L^U := \{w \in L \mid \supp(w) = U\}.
\]

\begin{remark}\label{rem:fullsuppregular}
Note that for any language $L \subseteq X_V^*$ and subset $U \subseteq V$,
\begin{equation}\label{eq:fullsupportlanguage} 
L^U = L \cap (\cap_{u \in U} X_V^*X_uX_V^*).
\end{equation}
Hence if $L$ is a regular language, then so is the language 
$L^U$.
\end{remark}

For any subset $U \subseteq V$, 
Proposition~\ref{prop:slsubset} shows that 
\[
\sphcl(G_V,X_V)^U=\sphcl(G_U,X_U)^U;
\] 
we also denote this set by $\sphcl^U$.
Thus this set can be viewed
as the set of shortlex conjugacy normal forms for the pair $(G_U,X_U)$
that have ``full support'', in the sense that
they include at least one letter from $X_v$ for every $v \in U$.
Then $\sphcl(G_V,X_V)$ is a disjoint union
$$
\sphcl(G_V,X_V) = \bigsqcup_{U \subseteq V} \sphcl^U.
$$
Consequently the spherical conjugacy growth series
satisfies
\begin{equation}\label{eq:decompsum}
\sphcs_{(G_V,X_V)} = F_{\sphcl(G_V,X_V)} = \sum_{U \subseteq V} F_{\sphcl^U};
\end{equation}
that is, $\sphcs_{(G_V,X_V)}$ is the sum of the generating
functions for the languages $\sphcl^U$.

\begin{definition}\label{def:decomp}
Let $\Gamma=(V,E)$ and let $\ll$ be a total
order of $V$.
The {\em $\ll$-decomposition} of a nonempty subset $U \subseteq V$ is the
ordered tuple 
$$
\decomp(U) = (U_1,U_2,...,U_m)
$$
such that
\begin{itemize}
\item[(a)] the connected components of the complement graph
$\dual{\Gamma_U}$ are the subgraphs induced by $U_1,...,U_m$, and
\item[(b)] for each $1 \le i <  m$, there is a vertex
$u \in U_i$ satisfying $u \ll v$ for all $v \in \cup_{j=i+1}^m~ U_j$.
\end{itemize}
A subset $U\subseteq V$ is called {\em indecomposable}, or {\em non-split}, 
if $U$ is nonempty and
the induced complement subgraph $\dual{\Gamma_U}$ has exactly
one connected component.
\end{definition}

This terminology comes from Remark~\ref{rmk:dirprod};
if $\decomp(U) = (U_1,U_2,...,U_m)$, then
$G_U \cong G_{U_1} \times \cdots \times G_{U_m}$ decomposes,
or splits, as a direct product of subgraph products.
Note that in the case that $U = \emptyset$, the graph $\dual{\Gamma_U}$
has no components;
hence $U$ is not called indecomposable.
\begin{example}
Let $\Gamma$ be the following graph, with $1\ll 2\ll 3\ll 4\ll 5.$
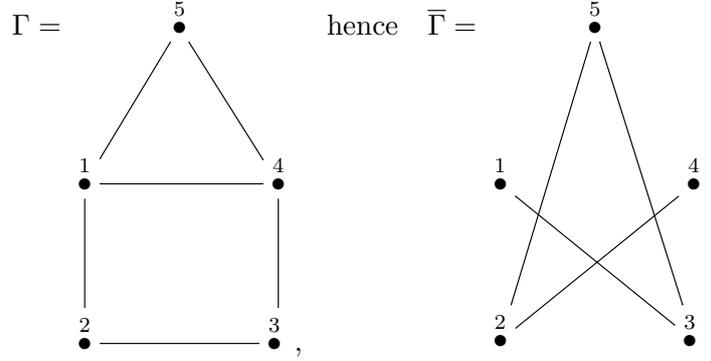
\[
  \Gamma=\xymatrix@R1.4cm{& \stackrel{5}{\bullet}\ar@{-}[ld]\ar@{-}[rd] &\\
  \stackrel{1}{\bullet} \ar@{-}[rr]\ar@{-}[d] && \stackrel{4}{\bullet}\ar@{-}[d]\\
  \stackrel{2}{\bullet}\ar@{-}[rr]&& \stackrel{3}{\bullet}\rlap{\,,}}
  \quad\textup{hence}\quad
  \dual{\Gamma}=\xymatrix@R1.4cm{& \stackrel{5}{\bullet}\ar@{-}[ldd]\ar@{-}[rdd] &\\
  \stackrel{1}{\bullet} \ar@{-}[rrd] && \stackrel{4}{\bullet}\ar@{-}[lld]\\
  \stackrel{2}{\bullet}&& \stackrel{3}{\bullet}\rlap{\,.}}
 \]

Then
\begin{eqnarray*}\decomp(\{1,2,3,4,5\}) &=& (\{1,2,3,4,5\}), \\ \decomp(\{1,2,3,4\})&=&(\{1,3\},\{2,4\}),\\  \decomp(\{1,3,4,5\})&=&(\{1,3,5\},\{4\}), \\ \decomp(\{1,2,4,5\})&=&(\{1\},\{2,4,5\}),\\
\decomp(\{1,4,5\})&=&(\{1\},\{4\},\{5\}).\\
\end{eqnarray*}
The computation of $\decomp(U)$ with $U\subseteq\{1,2,3,4,5\}$ different from the ones above will be done in Example \ref{ex:3edgelinesegment} (since they are isomorphic to a subgraph of a line consisting of three segments).
\end{example}

\begin{theorem}\label{thm:conjsldecompU}
Let $G_V$ be a graph product group with generating set $X_V$,
let $<_V$ be a total order of $X_V$ 
that is compatible with
a total order $\ll$ of $V$, and let $<_{sl}$ be the corresponding
shortlex order on $X_V^*$.
Let $U$ be a nonempty subset of $V$ and let 
$\decomp(U) = (U_1,...,U_m)$. Then
\begin{itemize}
\item[(1)] there is a well-defined bijection 
$\Psi:\sphcl^U \rightarrow \sphcl^{U_1} \cdots \sphcl^{U_m}$  
that preserves conjugacy class and word length, and
\item[(2)] $F_{\sphcl^U} = F_{\sphcl^{U_1}} \cdots F_{\sphcl^{U_m}}$.
\end{itemize}
\end{theorem}

\begin{proof}
Let $w$ be any word in $\sphcl^U$.
The definition 
of the $\ll$-decomposition $\decomp(U)$ of $U$
shows that 
whenever $v \in U_i$ and $v' \in U_j$ with $i \neq j$,
there is an edge $\{v,v'\}$ in the edge set $E$ of $\Gamma$, and
from Remark~\ref{rmk:dirprod} we have
$G_U = G_{U_1} \times \cdots \times G_{U_m}$.
Thus there is a sequence of shuffles 
from $w$ to a word of the form $w_1 \cdots w_m$ such that
$w_i \in X_{U_i}^*$ and $\supp(w_i)=U_i$ for each $i$,
and such that each shuffle commutes subwords in vertex
generating sets corresponding to vertices in distinct elements of 
the tuple $\decomp(U)$.
Although at each step
there may be more than one possible choice of shuffle operation
to be applied next, since no operation shuffles
two subwords within the same set $X_{U_i}^+$, 
the word $w_1 \cdots w_m$ is unique.  
Let $\Psi(w) := w_1 \cdots w_m$.

Since $\Psi(w)$ is obtained from $w$ using shuffle
operations, $l(\Psi(w)) = l(w)$ and $\Psi(w) =_{G_V} w$.

Let $1 \le i \le m$ and let $g_i \in G_{U_i}$.
Then 
$w \sim_{G_V} g_iwg_i^{-1} =_{G_V} g_i\Psi(w)g_i^{-1} =_{G_V}
w_1 \cdots w_{i-1}(g_iw_ig_i^{-1})w_{i+1} \cdots w_m$
since $g_i$ commutes with $w_j$ for all $i \ne j$, and so
$w \sim_{G_V} w_1 \cdots w_{i-1}x_iw_{i+1} \cdots w_m$ 
where $x_i$ is the shortlex normal form for $g_iw_ig_i^{-1}$.
Since $w \in \sphcl^{U} \subseteq \sphcl(G_V,X_V)$, the word $w$ is a conjugacy
geodesic, and so we have 
$l(w) \le l(w_1 \cdots w_{i-1}x_iw_{i+1} \cdots w_m)$
and hence $l(w_i) \le l(x_i)$. 
If $l(x_i) = l(w_i)$, we can write
$w_i = t_1 \cdots t_k$ and $x_i = t_1' \cdots t_k'$ with each $t_j,t_j' \in X_{U_i}$,
and $w=b_0t_1 b_1 \cdots t_kb_k$ with each $b_j \in (X_V \setminus X_{U_i})^*$.
In that case, again using the fact that every $t_j$ and $t_j'$ commutes
with all of the letters in $b_0b_1 \cdots b_k$ yields
$w \sim_{G_V} g_iwg_i^{-1} =_{G_V} b_0t_1' b_1 \cdots t_k'b_k$.
Since $w$ is a shortlex conjugacy normal form, we have
$w \le_{sl} b_0t_1' b_1 \cdots t_k'b_k$, and hence $w_i \le_{sl} x_i$.
Then $w_i \in \sphcl(G_{U_i},X_{U_i})$ and so Proposition~\ref{prop:slsubset}
says that $w_i \in \sphcl(G_{V},X_{V})$ as well.
Thus $w_i \in \sphcl^{U_i}$. 
Therefore the image of the map $\Psi$ lies in $\sphcl^{U_1} \cdots \sphcl^{U_m}$.

Suppose that $w'$ is another word in $\sphcl^U$ and $\Psi(w')=\Psi(w)$.
Then $w' =_{G_V} w$, and $w,w'$ are shortlex normal forms for the same
element of $G_V$. Hence $w'=w$ and $\Psi$ is injective.

Next suppose that $y_i \in \sphcl^{U_i}$ for each $i$,
and let $z \in \sphcl(G_V,X_V)$ be the shortlex conjugacy normal form for the
conjugacy class containing the element represented by $y_1 \cdots y_m$.
Proposition~3.5 of~\cite{CH14} shows that the word
$y_1 \cdots y_m$ over $X_V$ is a conjugacy geodesic.
Since $z$ and $y_1 \cdots y_m$ are conjugacy geodesics
representing the same conjugacy class,
Corollary~3.6 of~\cite{chm} says that $\supp(z) = \supp(y_1 \cdots y_m) = U$
and $z \sim_{G_U} y_1 \cdots y_m$.
Then $z \in \sphcl^U$ and
we can write $\Psi(z) = z_1 \cdots z_m$ with each $z_i \in \sphcl^{U_i}$.
Moreover, there is
an element $g \in G_U$ such that $z =_{G_V} gy_1 \cdots y_mg^{-1}$.
Using the direct product structure of $G_U$, then
$g=g_1 \cdots g_m$ for some elements $g_j \in G_{U_j}$
and 
$z_1 \cdots z_m =_{G_V} z =_{G_V} g_1y_1g_1^{-1} \cdots g_my_mg_m^{-1}$; 
hence 
$z_i =_{G_V} g_iy_ig_i^{-1}$ for each $i$.  Since 
$z_i$ and $y_i$ are both shortlex conjugacy normal forms,
then $z_i=y_i$ for all $i$.
Therefore $\Psi(z) = y_1 \cdots y_m$ and $\Psi$ is surjective,
completing the proof of part~(1).

Part~(2) follows from part~(1) and the fact that whenever a language
$L$ over an alphabet $X$ is a concatenation $L=L_1 \cdots L_m$ of languages 
$L_1,...,L_m$ over $X$, then $f_L = f_{L_1} \cdots f_{L_m}$.
\end{proof}

Note that when $U = \emptyset$,
$\sphcl^U = \{\emptyword\}$ and hence
$F_{\sphcl^U} = 1$.
The following is now an immediate consequence of 
Equation~(\ref{eq:decompsum})
and Theorem~\ref{thm:conjsldecompU}. 

\begin{theorem}\label{thm:gpconjsldecomp}
Let $G_V$ be a graph product group over a graph
with vertex set $V$. Let $X_V$ be
a union of inverse-closed generating sets for the
vertex groups,
let $<_V$ be a total order of $X_V$ 
that is compatible with
a total order $\ll$ of $V$, and let $<_{sl}$ be the corresponding
shortlex order on $X_V^*$.
Then
the language
\[
\{\emptyword\} \bigsqcup \left(\bigsqcup_{U \subseteq V,~\decomp(U)=(U_1,...,U_m)}
\sphcl^{U_1} \cdots \sphcl^{U_m}\right)
\] 
is a set of unique
conjugacy geodesic representatives for the conjugacy classes of $G_V$ over $X_V$,
and the spherical conjugacy growth series satisfies
\[
\sphcs_{(G_V,X_V)} = 1 + 
\sum_{\emptyset \ne U \subseteq V,~\decomp(U)=(U_1,...,U_m)}
F_{\sphcl^{U_1}} \cdots F_{\sphcl^{U_m}}.
\]
\end{theorem}


\section{RAAGs: 
Conjugacy representatives and a formula for $\sphcs$}
\label{sec:conjrepraag}


The goal of Section~\ref{sec:conjrepraag} is to build upon
the conjugacy geodesic representatives of conjugacy classes from 
Section~\ref{subsec:decompconjsl} in the case of right-angled
Artin groups, 
in order to produce another language of conjugacy geodesic representatives 
that naturally extends the language of cyclically reduced words in free groups 
(but is different from $\sphcl$), and to  
give a formula for the spherical conjugacy growth series of a RAAG.
This language of conjugacy geodesic representatives
also satisfies the property that its growth series can be easily obtained using the formulas in Section~\ref{subsec:cycrep}, provided that the growth series of a regular language built out of shortlex representatives is available; 
applications to computing the conjugacy growth
series for some families of RAAGs are discussed in 
Sections~\ref{subsec:raagconjreps} and~\ref{sec:examples}.

We first gather some more terminology and known
results on languages associated to RAAGs in Section~\ref{subsec:raaglanguages}.


\subsection{Further background on RAAGs and languages}\label{subsec:raaglanguages}


A {\em right-angled Artin group}, 
or {\em RAAG}, with respect to a finite simple graph
$\Gamma = (V,E)$ is the graph product group $G_V$ such that
 each vertex group $G_v$ 
is isomorphic to $\mathbb{Z}$.
If each vertex generating set
$X_v=\{x_v, x_v^{-1}\}$ consists of a generator of $G_v$
and its inverse, we refer to the generating set
$X_V = \cup_{v \in V} X_v$ as the 
{\em Artin generating set} for the RAAG $G_V$.

For RAAGs, two of the languages discussed
in this paper coincide.

\begin{proposition}\cite[Corollary~2.5]{CHHR14}\label{prop:cycgeoisconjgeo}
If $G_V$ is a right-angled Artin group, 
then a word $w$ over the Artin generating
set $X_V$ is a conjugacy geodesic 
if and only if every cyclic conjugate
of $w$ is geodesic; that is, $\geocl(G_V,X_V) = \geocpl(G_V,X_V)$.
\end{proposition}

Recall that a shuffle operation on a word $w \in X_V^*$ is 
an application of a commuting relation,
replacing a word $yabz$ by the word $ybaz$ whenever
$y,z \in X_V^*$,
$a \in X_v^+$, $b \in X_{v'}^+$, and $\{v,v'\}$
is an edge of the graph $\Gamma$.
A {\em deletion} exchanges a word of
the form $yay'a^{-1}y''$ by the word $yy'y''$
whenever $y,y',y'' \in X_V^*$, $a \in X_v^*$ for some $v \in V$,
and for each $v' \in \supp(y')$, either 
$v'=v$ or $v$ and $v'$ are joined by an edge in $\Gamma$.
We write $w \stackrel{sh}{\longrightarrow} w'$ if there is
a finite sequence of shuffle operations from $w$ to $w'$,
and $w \stackrel{shde}{\longrightarrow} w'$ if there is
a finite sequence of shuffle and deletion
operations from $w$ to $w'$.

\begin{proposition}\cite[p.~36]{serv},\cite[Prop~3.3,Cor3.4]{CH14}\label{prop:shuffle}
Let $G_V$ be a right-angled Artin group with respect to a
graph $\Gamma$, and let $X_V$ be the Artin generating set.
If $w$ is any word over $X_V$ and $z$ is the shortlex normal
form for $w$, then $w \stackrel{shde}{\longrightarrow} z$.
Moreover, if $w \in \geol(G_V,X_V)$, then 
$w  \stackrel{sh}{\longrightarrow} z$.
\end{proposition}

In Section~\ref{subsec:raagconjreps} we use several results
from~\cite{CGW}; since the 
terminology in~\cite{CGW} differs somewhat from ours,
we include a comparison here.

\begin{remark}\label{rmk:cgwtranslation}
A word that is geodesic (respectively, cyclically geodesic)
is called reduced (respectively, cyclically reduced) in~\cite{CGW}.
Write the vertex set $V = \{v_1,...,v_n\}$;
using a total order on $X_V$ compatible with the
order $v_n < v_{n-1} < \cdots < v_1$ on $V$,
a word that is in shortlex normal form (respectively, in $\sphcpl$)
is said to be in normal form (respectively, cyclic normal form) 
in~\cite{CGW}.  Finally, a word $w \in X_V^*$ for which 
$\supp(w)$ is indecomposable, or non-split, is referred to as
a non-split word in~\cite{CGW}.
\end{remark}


\subsection{Conjugacy geodesic representatives and spherical conjugacy growth for RAAGs}\label{subsec:raagconjreps}


We begin this section with Theorem~\ref{thm:conjrep}
and Corollary~\ref{cor:conjrepseries}
that establish the set of conjugacy geodesic
representatives which we will use for computing the 
spherical conjugacy growth series for RAAGs.
To accomplish this, we combine our setup from 
Sections~\ref{sec:preliminaries} and~\ref{sec:graphproduct} 
with results of Crisp, Godelle, and Wiest in~\cite{CGW}.

\begin{theorem}\label{thm:conjrep}
Let $G_V$ be a right-angled Artin group with Artin generating
set $X_V$, let $<_V$ be a total order of $X_V$
that is compatible with a total order $\ll$ of $V$, 
and let $<_{sl}$ be the corresponding shortlex order on $X_V^*$.
Let $U$ be a subset of $V$. 
If $U$ is indecomposable, then there is a well-defined bijection 
\[
\Phi: \cycrep(\sphcpl(G_V,X_V)^U) \rightarrow \sphcl(G_V,X_V)^U 
\]
that preserves conjugacy class and word length.
\end{theorem}

\begin{proof}
Suppose that $U$ is an indecomposable subset of $V$. 
For each element $w$ of $\cycrep(\sphcpl^U)$, define
$\Phi(w)$ to be the conjugacy shortlex normal form of 
the conjugacy class containing $w$.

Let $w$ be any element of $\cycrep(\sphcpl^U) \subseteq \sphcpl^U$.
Since every shortlex normal form is geodesic,
then $w \in \geocpl^U$ as well.
Now $\Phi(w) \in \sphcl(G_V,X_V) \subseteq \geocl(G_V,X_V)$,
and so Proposition~\ref{prop:cycgeoisconjgeo} says
that $\Phi(w) \in \geocpl(G_V,X_V)$ also. Now
Lemma~2.9 of~\cite{CGW} says that the two cyclically geodesic
words $w$ and $\Phi(w)$, which represent conjugate elements of $G_V$,
are related by a sequence of cyclic permutations and commutation
relations (i.e., shuffles). Hence these two words have the
same support, and so $\Phi(w) \in \sphcl(G_V,X_V)^U$, as required.

Applying Proposition~\ref{prop:cycgeoisconjgeo} to the word $w$ shows that
$w$ is a conjugacy geodesic, as is the word $\Phi(w)$, 
and hence the function $\Phi$ preserves word length.

Now suppose that $w, w' \in \cycrep(\sphcpl^U)$ and $\Phi(w)=\Phi(w')$.
Then $w$ and $w'$ are cyclically shortlex normal forms
that represent elements in the same conjugacy class,
and so~\cite[Proposition~2.21]{CGW} shows that
$w'$ is a cyclic permutation of $w$. But since both words
are lexicographically least among their common set of cyclic permutations,
$w=w'$ and $\Phi$ is injective.

Finally,~\cite[Propositions~2.18 and~2.20]{CGW} show
that every word in $\geocpl^U$, and hence every
word in $\sphcl^U$, 
represents an element of $G_V$ conjugate to an element 
represented by a cyclically shortlex word. 
Therefore $\Phi$ is also surjective.
\end{proof}

The following is now immediate from
Theorem~\ref{thm:conjrep} and
Theorem~\ref{thm:gpconjsldecomp}. 

\begin{corollary}\label{cor:conjrepseries}  
Let $G_V$ be a right-angled Artin group with Artin generating
set $X_V$, let $<_V$ be a total order of $X_V$
that is compatible with a total order $\ll$ of $V$, 
and let $<_{sl}$ be the corresponding shortlex order on $X_V^*$.
Then
the set 
\[
\{\emptyword\} \bigsqcup 
\left(\bigsqcup_{\emptyset \ne U \subseteq V,~\decomp(U)=(U_1,...,U_m)}
\cycrep(\sphcpl^{U_1}) \cdots \cycrep(\sphcpl^{U_m})\right)
\]
is a language of unique conjugacy geodesic representatives
for the conjugacy classes of $G_V$ over $X_V$
, and
\begin{eqnarray*}
\sphcs_{(G_V,X_V)} &=& 
1 +  \sum_{\emptyset \ne U \subseteq V,~\decomp(U)=(U_1,...,U_m)}
F_{\cycrep(\sphcpl^{U_1})} \cdots F_{\cycrep(\sphcpl^{U_m})} . 
\end{eqnarray*}
\end{corollary}


\begin{remark}\label{rmk:languagesnotequal}
We note that in general the languages of conjugacy geodesic conjugacy class
representatives for RAAGs in
Theorem~\ref{thm:gpconjsldecomp} and
Corollary~\ref{cor:conjrepseries} 
are not the same.
For example, suppose
that the defining graph
$\Gamma=(V,E)$ of the RAAG $G_V$ satisfies $V=\{a,b,c\}$,
the edge set $E$ contains only one edge, joining $a$ and $c$,
and the ordering on $X_V$
is given by $x_a<x_a^{-1}<x_b<x_b^{-1}<x_c<x_c^{-1}$.
Then for the indecomposable set $U=V$, the word
$w = x_ax_cx_b$ lies in $\cycrep(\sphcpl(G_V,X_V)^U)$,
but $w =_{G_V} x_cx_ax_b$, and the cyclic conjugate word
$x_ax_bx_c$ is the element of $\sphcl(G_V,X_V)^U$
representing the conjugacy class of $w$.
\end{remark}

In order to use Corollary~\ref{cor:conjrepseries} to compute
the spherical conjugacy growth series of $G_V$
over $X_V$, we need 
to be able to find the growth
series for $\cycrep(\sphcpl^{U})$ when $U$ is indecomposable.
Proposition~\ref{prop:kthpower} allows us to
use Corollary~\ref{cor:decycle} in this computation.

\begin{proposition}\label{prop:kthpower}
Let $G_V$ be a right-angled Artin group on a finite simple
graph $\Gamma=(V,E)$ with Artin generating
set $X_V$, let $<_V$ be a total order of $X_V$
that is compatible with a total order $\ll$ of $V$, 
and let $<_{sl}$ be the corresponding shortlex order on $X_V^*$.
If $U$ is an indecomposable subset of $V$, then

(1) $(\cycsl^U)^{\power k}\subseteq \cycsl^U$ for all $k\geq 1$, and 

(2) $\cycperm(\cycsl^U)=\cycsl^U$.
\end{proposition}

\begin{proof}
Let $w \in \cycsl^U$.  For any cyclic permutation $w'$
of $w$, the support sets of $w'$ and $w$ are the same,
and the sets of
cyclic permutations of $w'$ and $w$ are the same and hence
consist of shortlex normal forms. Therefore property~(2) holds.

To prove property~(1), 
we first consider the case that $|U| = 1$, and write 
$X_U = \{x_v,x_v^{-1}\}$ for the unique vertex $v \in U$. 
Using
the fact from Proposition~\ref{prop:slsubset} that
$\cycsl^U=\cycsl(G_V,X_V)^U = \cycsl(G_U,X_U)^U$ is
the set of cyclically shortlex normal forms in
the abelian group $G_U$, 
the language $\cycsl^U$ is the set of words of the form
$x_v^m$ or $(x_v^{-1})^m$ with $m \ge 1$,
and the language $(\cycsl^U)^{\power k}$
is the subset of $\cycsl^U$ consisting of  positive powers of
$x_v^k$ and $(x_v^{-1})^k$, as needed. 
 
For the remainder of the proof we assume that $U$ is an
indecomposable subset of $V$ satisfying $|U| \ge 2$.
Since $\supp(w) = U$, we also have $\supp(w^k) = U$.
Let $z$ be the shortlex normal form for $w^k$.
By Proposition~\ref{prop:shuffle}, there is a finite sequence of
shuffles and deletions from $w^k$ to $z$.

Suppose that at least one deletion occurs in this sequence. Then there
is a sequence of shuffles from $w^k$ to a word of the
form $z' := yay'a^{-1}y''$ for some
$y,y',y'' \in X_U^*$ and $a \in X_v$ with $v \in U$,
such that for all $v' \in \supp(y')$ either $v' = v$
or $v,v'$ are joined by an edge in $\Gamma$. 
We can write $w^k=w_1aw_2a^{-1}w_3$ with each $w_i \in X_U^*$,
and since only shuffles and no deletions are performed
in transforming $w^k$ to $z'$, all of the letters of 
$w_2$ must either be a letter of $y'$ or commute via a shuffle with 
one (and hence both) of the letters $a$ or $a^{-1}$. 
Now $a,a^{-1}$ are in the same vertex generating set $X_v$,
and for each vertex $v' \in \supp(w_2)$, either $v' = v$ or
there is an edge connecting $v$ and $v'$
in $\Gamma$.
Since $U$ is indecomposable, the complement $\dual{\Gamma_U}$
of the induced graph $\Gamma_U$
is connected, and since $|U| \ge 2$, there
is an edge in $\dual{\Gamma_U}$ from $v$ to another vertex $v''$.
Now $v'' \ne v$
and $v''$ is not connected to $v$ by an
edge in $\Gamma$. 
Hence $\supp(w_2) \ne U$, and so
$w_2$ cannot contain a cyclic permutation of $w$ as a subword;
that is, $w_2$ is shorter than the word $w$.
Also, since $w$ is a shortlex normal form, the word $aw_2a^{-1}$
cannot be a subword of $w$.  Then we can write 
$w=caf = ga^{-1}d$ for some $c,f,g,d \in X_U^*$ such that
$fg = w_2$ and $\supp(c),\supp(d)$ are not contained in
$\supp(w_2) \cup \{v\}$.
Consequently $af$ must be a suffix of $d$ and we can write
$w = ga^{-1}haf$ for some $h \in X_U^+$. However, then
the word $afga^{-1}h$ is a cyclic permutation of the
word $w$ that is not a geodesic, contradicting the fact that 
$w$ is cyclically shortlex.
Hence no deletion can occur in transforming $w^k$ to its
shortlex normal form $z$, and so 
there is a sequence of shuffles
from $w^k$ to $z$ and $w^k$ is a geodesic. 

Suppose next that $w^k$ is not in shortlex normal form 
for some $k > 1$,
and let $m > 1$ denote the least integer greater than 1
for which $w^m \notin \sphl$.
Let $z$ be the shortlex normal form for the group
element represented by $w^m$.
Let $f$ be the longest common prefix of the words $w^m$ and $z$,
and write $w^m= fw'$ and $z = fz'$.
Since $w^m \ne z$, then $w' \ne z'$, and since
$z$ is in shortlex normal form, the word $z'$
is the shortlex normal form for $w'$ and $z' <_{sl} w'$.
Then by our choice of $m$, the word $w'$ cannot be a subword
of $w^k$ for any $k<m$, and so the length of $f$ is strictly less
than the length of $w$. Write $w=fg$ with $g \in X_U^+$,
and let $\widetilde w := gf$.
 
An earlier part of this proof shows that $w^m$ is geodesic;
hence the subword $w'$ is also geodesic.
Since $z' <_{sl} w'$ and the words $w'$ and $z'$ have no common prefix,
we have $z'=az''$ and $w' = bw''$
for some $a,b \in X_U$ and $z'',w'' \in X_U^*$ such that 
$a <_V b$. Let $v$ be the vertex for which $a \in X_v$.

Proposition~\ref{prop:shuffle} says
that there is a sequence of shuffles from $w'$ to $z'$, and
the same sequence of shuffles transforms the word 
$\widetilde w^m = w'f = bw''f$
to $z'f = az''f$.
Thus we can write $\widetilde w^m = bw_1aw_2$ with $w_1,w_2 \in X_U^*$
such that $a$ commutes via shuffles with all of the letters of $bw_1$.
That is, for all $v' \in \supp(bw_1)$,
the vertices $v$ and $v'$ are connected by an edge in $\Gamma$.
As noted above, since $U$ is indecomposable and $|U| \ge 2$,
there is a vertex $v'' \in U = \supp(\widetilde w)$ 
with $v'' \notin \supp(bw_1) \cup \{v\}$.
Hence $bw_1$ is shorter than the word $\widetilde w$,
and so $bw_1a$ is a prefix of the cyclic permutation
$\widetilde w$ of $w$ that is
not in shortlex normal form. This contradicts the
fact that the word $w$ is cyclically shortlex.
Therefore $w^k$ is in shortlex normal form for all $k \ge 1$.

For every $k \ge 1$,
every cyclic permutation $y$ of $w^k$ has the
form $y = \widetilde w^k$ for some cyclic permutation
$\widetilde w$ of $w$. 
Since $\widetilde w \in \cycsl^U$, the 
argument above applied to $\widetilde w$ shows that
$y = \widetilde w^k$ is in $\sphl$.
Therefore $w^k \in \sphcl^U$, completing the proof of~(1).
\end{proof}

Corollary~\ref{cor:decycle}
and Proposition~\ref{prop:kthpower}
yield the following.

\begin{corollary}\label{cor:cycrepseries}
Let $G_V$ be a right-angled Artin group on a graph $\Gamma=(V,E)$
and let $X_V$ be the Artin generating set.
Then for every indecomposable subset $U \subseteq V$,
$F_{\cycrep(\sphcpl^{U})} = \crseries(F_{\sphcpl^{U}})$.
\end{corollary}

We now have all of the ingredients for our closed
formula for the spherical conjugacy growth series for
RAAGs;
Theorem~\ref{thm:formula} follows immediately from
Corollaries~\ref{cor:conjrepseries} and~\ref{cor:cycrepseries}.

\begin{theorem}\label{thm:formula}
Let $G_V$ be a right-angled Artin group on a graph $\Gamma=(V,E)$
and let $X_V$ be the Artin generating set.
The spherical conjugacy growth series satisfies
the formula
\begin{eqnarray*}
\sphcs_{(G_V,X_V)} &=& 
1 +  \sum_{\emptyset \ne U \subseteq V,~\decomp(U)=(U_1,...,U_m)}
\crseries(F_{\sphcpl^{U_1}}) \cdots \crseries(F_{\sphcpl^{U_m}}). 
\end{eqnarray*}
where
for any formal complex power series $f$ with
integer coefficients and constant term equal to 0, 
\begin{equation*}
\crseries(f)(z) := 
\int_0^z \frac{\sum_{k\geq 1} \phi(k) f(t^k)}{t} \,dt.
\end{equation*}
\end{theorem}

\begin{remark}\label{rmk:freeranknspherical}
The free group on n generators 
$F_n = G_V = \langle a_1,...,a_n \mid ~ \rangle$, 
is the RAAG  with defining graph
 \[
 \Gamma = \Gamma_V = \hspace{.1in}
 \stackrel{a_1}{\bullet} \hspace{.3in} \stackrel{a_2}{\bullet} 
 \hspace{.3in} \cdots \hspace{.3in} \stackrel{a_n}{\bullet} . 
 \] 
Since every nonempty subset $U$ of $V$ is indecomposable,
the formula in Theorem~\ref{thm:formula} is 
\begin{eqnarray*}
\sphcs_{(F_n,X_V)} 
&=& 
1 +  \sum_{\emptyset \ne U \subseteq V}
\crseries(F_{\sphcpl^{U}})
=
1 +  \crseries\left(\sum_{\emptyset \ne U \subseteq V}
F_{\sphcpl^{U}}\right)
\end{eqnarray*}
where the last equality uses
the fact that the operator $\crseries$ commutes with addition
(Remark~\ref{rmk:rhoisadditive}).
The cyclically shortlex language $\sphcpl$ is the disjoint
union 
\begin{equation}\label{eq:disjointunion}
\sphcpl = \{1\} ~\bigsqcup~ 
\left(\bigsqcup_{\emptyset \ne U \subseteq V} \sphcpl^{U}\right),
\end{equation}
and so $F_{\sphcpl} = 1 + \sum_{\emptyset \ne U \subseteq V}
F_{\sphcpl^{U}}$. This recovers the formula
\begin{eqnarray*}
\sphcs(F_n,X_V) 
&=&
1 +  \crseries\left(F_{\sphcpl}-1\right)
\end{eqnarray*}
from Rivin's work~\cite{R10} in 
Equation~(\ref{eq:freegroup}) of Theorem~\ref{thm:rivincorrected} above;
Rivin's formula for $F_{\sphcpl}-1$ as a rational function
is also given in Theorem~\ref{thm:rivincorrected}.

For comparison with the formula in Equation~(\ref{eq:necklaceformula}) 
in our earlier paper~\cite[Theorem~A]{chm}
for the spherical
conjugacy growth series of the free group of rank $n$,
that result, together with induction on $n$, yields
\begin{eqnarray*}
\sphcs_{(F_n,X_V)} &=& \frac{1+(2n-1)z}{1-z} + 
\sum_{j=1}^{n-1} \neck\left(\frac{4jz^2}{(1-z)(1-(2j-1)z)}\right).
\end{eqnarray*}
where the operator $\neck$ is defined 
in Equation~(\ref{eq:transformN}).
\end{remark}

Next we apply Theorem~\ref{thm:formula} to give a formula
for the spherical conjugacy growth series of 
the free product of a RAAG $G$ with
a free group, in terms of the spherical conjugacy growth of the group $G$
along with growth of the cyclically shortlex languages for $G$ and the 
free product. 
In Example~\ref{ex:zstarzn}
in Section~\ref{sec:examples}, we use the formula
in Corollary~\ref{cor:freeprodwithfree} along with
computations in GAP of the growth series for 
the cyclic shortlex conjugacy languages to analyze the spherical
conjugacy growth series of the free product of $\mathbb{Z}$
with free abelian groups $\mathbb{Z}^n$ for small $n$.

\begin{corollary}\label{cor:freeprodwithfree}
Let $G_{W}$ be a right-angled Artin group on a simple graph $\Gamma_W=(W,E)$,
and let $G_W \ast F_m$ be the free product of $G_{W}$ with
a free group of rank $m \ge 1$ with free basis $W'$. 
Viewing $G_W \ast F_m$ as a RAAG 
with defining graph $\Gamma_V := (V,E)$ where $V := W \sqcup W'$,
let $X_{W}$ and $X_V$ be the Artin generating sets
of $G_W$ and $G_W \ast F_m$, respectively.
Then 
\[
\sphcs_{(G_W \ast F_m,X_V)} = 
\sphcs_{(G_W,X_W)} +
\crseries\left(F_{\sphcpl(G_W \ast F_m,X_V)} - F_{\sphcpl(G_W,X_W)} \right).
\]
\end{corollary}

\begin{proof}
In the complement graph $\dual{\Gamma_V}$,
there are edges between every pair of vertices in $W'$, and between
every pair of vertices $u,v$ with $u \in W$ and $v \in W'$.
Hence for 
each subset $U \subset V$ satisfying $U \cap W' \ne \emptyset$,
the subgraph
of the complement graph induced by $U$, $\dual{\Gamma}_U$, is 
connected, and so $\decomp(U)=(U)$ in this case.
The formula in Theorem~\ref{thm:formula} now
says that the spherical
conjugacy growth series for $G_W \ast F_m$ with respect to $X_{V}$
satisfies
\begin{eqnarray}
\sphcs_{(G_W \ast F_m,X_{V})} &=& 
   1 + \sum_{\emptyset \ne U \subseteq W, \decomp(U)=(U_1,...,U_m)}
      \left(\crseries\left(F_{\sphcpl^{U_1}}\right) \cdots \crseries\left(F_{\sphcpl^{U_m}}\right)\right) \nonumber \\
&&   + \sum_{U \subseteq V,~U \cap W' \ne \emptyset} 
           \crseries\left(F_{\sphcpl^{U}}\right) \label{eq:freeprod1}.
\end{eqnarray}

Proposition~\ref{prop:slsubset} shows that for
any subset $U \subseteq W$,
$\sphcpl(G_W \ast F_m,X_V)^{U} = \sphcpl(G_W,X_W)^{U}$,
and so applying Theorem~\ref{thm:formula} to
the first two summands in Equation~(\ref{eq:freeprod1})
yields
\begin{equation}\label{eq:freeprod2}
1 + \sum_{\emptyset \ne U \subseteq W, \decomp(U)=(U_1,...,U_m)}
      \left(\crseries\left(F_{\sphcpl^{U_1}}\right) \cdots \crseries\left(F_{\sphcpl^{U_m}}\right)\right)
= \sphcs_{(G_W,X_W)}.
\end{equation}

Applying the additivity of the operator $\crseries$ from
Remark~\ref{rmk:rhoisadditive} along with Equation~(\ref{eq:disjointunion})
to the third summand in Equation~(\ref{eq:freeprod1}) yields 
\begin{eqnarray}
\sum_{U \subseteq V,~U \cap W' \ne \emptyset} 
           \crseries\left(F_{\sphcpl^{U}}\right)
&=&  \crseries\left(\sum_{U \subseteq V,~U \cap W' \ne \emptyset} 
           F_{\sphcpl^{U}}\right)  \nonumber \\
&=&  \crseries\left(F_{\sphcpl \cap (X_V^*X_{W'}X_V^*)}\right) 
\label{eq:freeprod3}
\end{eqnarray}
where $\sphcpl \cap (X_V^*X_{W'}X_V^*)$ is the set
of all cyclically shortlex words over $X_{V}$ that contain at least
one letter in the generating set $X_{W'}$ of $F_m$.
The language $\sphcpl=\sphcpl(G_W \ast F_m,X_V)$ is the disjoint union
of the  set $\sphcpl \cap (X_{V}^* X_{W'} X_{V}^*)$
and the set $\sphcpl(G_W \ast F_m,X_{V}) \cap X_{W}^*$.
Using Proposition~\ref{prop:slsubset} again, the latter set is
$\sphcpl(G_W,X_W)$, the set of cyclic shortlex
words for the RAAG subgroup $G_W$ of $G_W \ast F_m$ generated
by $X_W$. 
Then
\begin{equation}\label{eq:freeprod4}
F_{\sphcpl(G_W \ast F_m,X_V) \cap (X_{V}^* X_{W'} X_{V}^*)} = 
F_{\sphcpl(G_W \ast F_m,X_{V})}-
F_{\sphcpl(G_W,X_W)}.
\end{equation}
Combining Equations~(\ref{eq:freeprod1}),~(\ref{eq:freeprod2}),~(\ref{eq:freeprod3}), and~(\ref{eq:freeprod4})
completes the proof.
\end{proof}

Next we consider a RAAG that does not decompose as
a free product of sub-RAAGs.

\begin{example}\label{ex:3edgelinesegment}
Let $G=\langle a,b,c,d \mid [a,b]=[b,c]=[c,d]=1 \rangle$; that is,
$G=G_V$ is the RAAG with defining graph 
 \[
 \Gamma = \Gamma_V = 
 \xymatrix{\stackrel{a}{\bullet}\ar@{-}[r] &\stackrel{b}{\bullet} \ar@{-}[r]&\stackrel{c}{\bullet}\ar@{-}[r] &\stackrel{d}{\bullet} }
 \] 
with vertex set $V = \{a,b,c,d\}$ and
Artin (inverse-closed) generating set
$X_V=\{a^{\pm 1},b^{\pm 1},c^{\pm 1},d^{\pm 1}\}$.
(We simplify notation here by replacing $x_a$ by $a$, etc.)
The complement graph is
 \[
 \dual{\Gamma} = \dual{\Gamma_V} = 
 \xymatrix{\stackrel{c}{\bullet}\ar@{-}[r] &\stackrel{a}{\bullet} \ar@{-}[r]&\stackrel{d}{\bullet}\ar@{-}[r] &\stackrel{b}{\bullet} }.
 \] 
Isomorphic induced subgraphs of $\dual{\Gamma}$ define 
isomorphic Artin subgroups (with the isomorphism mapping
one Artin generating set to another), and hence give
rise to the same contribution to the spherical conjugacy growth series.
Consequently the formula in Theorem~\ref{thm:formula} yields
\begin{eqnarray*}
\sphcs_{(G_V,X_V)} 
&=& 1 + 
4\crseries\left(F_{\sphcpl^{\{a\}}}\right)
+3\crseries\left(F_{\sphcpl^{\{a\}}}\right)\crseries\left(F_{\sphcpl^{\{b\}}}\right)
\\ &&
+3\crseries\left(F_{\sphcpl^{\{a,c\}}}\right) 
+3\crseries\left(F_{\sphcpl^{\{a,c\}}}\right)\crseries\left(F_{\sphcpl^{\{b\}}}\right)
\\ &&
+3\crseries\left(F_{\sphcpl^{\{a,c,d\}}}\right)
+\crseries\left(F_{\sphcpl^{\{a,b,c,d\}}}\right).
\end{eqnarray*}
The language $\sphcpl^{\{a\}}=\sphcpl(G_V,X_V)^{\{a\}}
=\sphcpl(G_{\{a\}},X_{\{a\}})^{\{a\}}$
(using Proposition~\ref{prop:slsubset})
is the set of all cyclic shortlex normal forms
in the Artin group $\mathbb{Z}$ except for the empty word;
that is, 
$\sphcpl^{\{a\}}=  \{a^n \mid n \ge 1\}
\cup \{(a^{-1})^n \mid n \ge 1\}$.
Moreover, each of the words in the language
$\sphcpl^{\{a\}}$ is also lexicographically least
among all of its cyclic permutations, since
cylically permuting a word in $\sphcpl^{\{a\}}$ does
not alter the word; hence
$\cycrep(\sphcpl^{\{a\}})=\sphcpl^{\{a\}}$.
Using Corollary~\ref{cor:cycrepseries} now shows that
$\crseries(\sphcpl^{\{a\}}) = F_{\cycrep(\sphcpl^{\{a\}})} = \sum_{n=1}^\infty 2z^n
= \frac{2z}{1-z}$.
Then
\begin{eqnarray}
\nonumber \sphcs_{(G_V,X_V)} 
&=&
\frac{(1+z)(1+5z)}{(1-z)^2} +
\left(\frac{3(1+z)}{1-z}\right) \crseries\left(F_{\sphcpl^{\{a,c\}}}\right)
\\ &&
+3\crseries\left(F_{\sphcpl^{\{a,c,d\}}}\right)
+\crseries\left(F_{\sphcpl^{\{a,b,c,d\}}}\right).
\label{eq:3edgesegment1}
\end{eqnarray}
In Section~\ref{sec:examples}, we discuss a method for
computing the growth series
$F_{\sphcpl^{\{a,c\}}}$, $F_{\sphcpl^{\{a,c,d\}}}$,
and $F_{\sphcpl^{\{a,b,c,d\}}}$ as rational functions.
\end{example}


\section{Some computations of 
spherical and geodesic conjugacy growth series for RAAGs}\label{sec:examples}  


In this section we discuss using formal language theory
methods, along with the software package 
KBMAG~\cite{kbmag} in GAP~\cite{GAP4}, to
compute examples of spherical and geodesic conjugacy growth series
for a right-angled Artin group $G_V$ with respect to the
(inverse-closed) Artin generating set $X_V$.
In particular, we describe a method to compute 
the growth series of the
language $\sphcpl^U$ for each indecomposable subset $U \subseteq V$.

Note that 
$$
\cycsl(G_V,X_V) = X^*_V \setminus \cycperm(X^*_V\setminus \sphl(G_V,X_V) )
$$ 
and combining this with Equation~(\ref{eq:fullsupportlanguage}) yields 
\begin{eqnarray}
\nonumber \cycsl^U &=& 
\left(X^*_V \setminus \cycperm(X^*_V\setminus \sphl(G_V,X_V) )\right)
\\ &&
\cap \left(\cap_{u \in U} X_V^*X_uX_V^*\right) 
\cap \left(X_V^* \setminus 
  \left(\cap_{v \notin U} X_V^*X_vX_V^*\right)\right).\label{eq:cycslregular} 
\end{eqnarray}

For every right-angled Artin group $G_V$ the set
$\sphl(G_V,X_V)$ is a regular language~\cite{HM95,vanwyk}.
As regular languages are closed under complementation 
and cyclic permutations (see for example Lemma~2.1 in~\cite{CHHR14}),
Equation~(\ref{eq:cycslregular}) shows that
the language $\cycsl^U$ is also regular.

Every RAAG $G_V$ has an automatic structure
with respect to a shortlex order on $X_V^*$~\cite{HM95,vanwyk},
and a finite state automaton for the language
$\sphl(G_V,X_V)$ can be obtained by computing this
automatic structure with the KBMAG package~\cite{kbmag}
in the computational algebra system GAP~\cite{GAP4}. 
Commands are also available in the GAP KBMAG package 
for manipulating FSAs, and in particular
producing FSAs that accept the complement of a language or
the intersection of the languages of two other FSAs.
An automaton recognizing the regular language $X_V^* \setminus \sphl$
can be computed from this, and then a further
FSA recognising the language 
$\cycperm(X^*_V\setminus \sphl(G_V,X_V) )$
can be produced following the construction 
in the proof of \cite[Lemma 2.1]{CHHR14}.
For each $v \in V$, the set $X_V^* X_v X_V^*$
of words containing at least one letter from the
vertex generating set $X_v$ is
the language of a finite state automaton 
with state set $Q=\{q_1,q_2\}$ in which
$q_1$ is the initial state, $q_2$ is the only
accept state, and the
transition function $\delta:Q \times X_V \rightarrow Q$
satisfies $\delta(q_1,x_v^{\pm 1})=q_2$, 
$\delta(q_1,a) = q_1$ for all $a \in X_{V -\{v\}}$,
and $\delta (q_2,b)=q_2$ for all $b \in X_V$.
When these automata are input into in GAP,
the intersection
and complementation operations 
enables computation of the automaton whose language
is the set $\cap_{u \in U} X_U^*X_uX_U^*$ 
of words with support equal to $U$ for any $U \subseteq V$.
Hence an automaton for $\cycsl^U$ can be
constructed in GAP
following the formula in Equation~(\ref{eq:cycslregular}).
Finally, the growth series of the regular language accepted by an FSA, 
expressed as a rational function, can be obtained in KBMAG/GAP by 
applying the \texttt{GrowthFSA} command.
A sample of GAP code to compute the growth
function for $\cycsl^U$ is given
in Example~\ref{ex:gapcodecycsl}.

A similar algorithm can be used to compute the conjugacy
geodesic growth series for any RAAG.
As noted in Proposition~\ref{prop:cycgeoisconjgeo}, the
languages $\geocpl$ of cyclically geodesic words and $\geocl$
of conjugacy geodesic words coincide for RAAGs, so it
suffices to compute the growth of $\geocpl$.
This begins with the construction of a finite state automaton whose
language is $\geol$, as follows.
For each vertex $v \in V$,
let $L_v$ be the set
of words over $X_V$
for which no sequence of shuffles results in
a subword of the form $x_vx_v^{-1}$ or
$x_v^{-1}x_v$. Then $L_v$
is the language of a finite state automaton 
with state set $Q=\{q_1,q_2,q_2,q_4,F\}$ in which
$q_1$ is the initial state, $q_1,q_2,q_3,q_4$ are
accept states, and the transition function 
$\delta:Q \times X_V \rightarrow Q$ satisfies 
\[
\begin{array}{llll}
\delta(q_1,x_v)=q_2 & \delta(q_1,x_v^{-1})=q_3 
    & \delta(q_1,b) = q_1 & \delta(q_1,c)=q_4 \\
\delta(q_2,x_v)=q_2 & \delta(q_2,x_v^{-1})=F 
    & \delta(q_2,b) = q_2 & \delta(q_2,c)=q_4 \\
\delta(q_3,x_v)=F & \delta(q_3,x_v^{-1})=q_2 
    & \delta(q_3,b) = q_3 & \delta(q_3,c)=q_4 \\
\delta(q_4,x_v)=q_2 & \delta(q_4,x_v^{-1})=q_3 
    & \delta(q_4,b) = q_4 & \delta(q_4,c)=q_4 \\
\end{array} \]
for all $b \in X_{V \setminus \{v\}}$
that commute with $x_v$ and all 
$c  \in X_{V \setminus \{v\}}$ that do
not commute with $x_v$,
and $\delta(F,d)=F$ for all $d \in X_V$.
(Note: If for all $u \in V \setminus \{v\}$
the vertex $u$ is connected to $v$ by an edge in $\Gamma_V$,
then the state $q_4$ is not needed for this FSA.)
Loading these automata into GAP 
allows computation of the language $\geol$,
since Proposition~\ref{prop:shuffle} shows that
$\geol = \cap_{v \in V} L_V$.
Now Proposition~\ref{prop:cycgeoisconjgeo}
and the analog of the formula in Equation~(\ref{eq:cycslregular}) give
\begin{equation}\label{eq:cycgeoregular}
\geocl(G_V,X_V) = \geocpl(G_V,X_V) = 
X^*_V \setminus \cycperm(X^*_V\setminus \geol(G_V,X_V) ),
\end{equation}
and the functions in GAP for computing $\cycsl$ also
apply to compute $\geocl$ for RAAGs.
Again using the \texttt{GrowthFSA} command in GAP results
in a rational expression for the conjugacy geodescis growth series.
A sample of GAP code to compute the conjugacy geodesic growth
series for $\geocl$ is given
in Example~\ref{ex:gapcodecycgeo}.

\medskip

In the following examples, we apply both algorithms 
to some families of RAAGs.

\begin{example}\label{ex:zstarzn}
The free product $H_n := \mathbb{Z} \ast \mathbb{Z}^n$ of an infinite
cyclic group with the free abelian group $J_n \cong \mathbb{Z}^n$ on $n$ generators
is the RAAG 
$$
H_n = \langle a,b_1,...,b_n \mid [b_i,b_j] = 1 \text{ for all }i,j \rangle
$$
with defining graph $\Gamma$ whose vertex set 
is $V_n = W_n \sqcup W'$ where $W_n = \{b_1,...,b_n\}$ and
$W' = \{a\}$, with an edge between $b_i$ and $b_j$ for all $i \ne j$.
Corollary~\ref{cor:freeprodwithfree}
says that the spherical
conjugacy growth series for $H_n$ with respect to 
the Artin generating set
$X_{V_n} = \{a^{\pm 1},b_1^{\pm 1},\ldots,b_n^{\pm 1}\}$
is
\begin{eqnarray*}
\sphcs_{(H_n,X_{V_n})} &=& 
\sphcs_{(J_n,X_{W_n})} +
\crseries\left(F_{\sphcpl(H_n,X_{V_n})} - F_{\sphcpl(J_n,X_{W_n})} \right)
\end{eqnarray*}

The spherical conjugacy growth series of the free
abelian group $J_n$ is
the product of the spherical conjugacy
growth series of the $n$ vertex groups; that is,
$
\sphcs_{(J_n,X_W)} = ((1+z)/(1-z))^n.
$
A word $w$ in $X_{W_n}^*$ containing two distinct letters
$b,b' \in X_{W_n}$
cannot be a cyclic shortlex normal form, since the
group $J_n$ is abelian and there is a cyclic
permutation of $w$ in which $b$ precedes $b'$ and
another cyclic permutation of $w$ in which $b'$
precedes $b$. Hence 
$\sphcpl(J_n,X_{W_n}) = \{b_i^k \mid 1 \le i \le n, k \in \mathbb{Z}\}$,
and so the growth series of this language is
$F_{\sphcpl(J_n,X_{W_n})} = (1+(2n-1)z) / (1-z)$. 
Plugging this into the formula for $\sphcs$ above yields
\begin{eqnarray*}
\sphcs_{(H_n,X_{V_n})} 
&=&  \left(\frac{1+z}{1-z}\right)^n +
\crseries\left(F_{\sphcpl(H_n,X_{V_n})} - \frac{1+(2n-1)z}{1-z}\right) .
\end{eqnarray*}

The GAP code described above (and in Example~\ref{ex:gapcodecycsl})
shows that
for $n \le 5$, the formula for the spherical
conjugacy growth series of $H_n = \mathbb{Z} \ast \mathbb{Z}^n$ 
is given by 
\begin{eqnarray*}
\sphcs_{(H_1,X_{V_1})} &=& 
    \frac{1+z}{1-z} +
    \crseries\left(\frac{2z(1+3z)}{(1+z)(1-3z)}\right), \\
\sphcs_{(H_2,X_{V_2})} &=&  
    \left(\frac{1+z}{1-z}\right)^2 +
    \crseries\left(\frac{2z(1+5z-z^2-z^3)}
         {(1+z)(1-z)(1-4z-z^2)}\right), \\
\sphcs_{(H_3,X_{V_3})} &=& 
    \left(\frac{1+z}{1-z}\right)^3 +
    \crseries\left(\frac{2z(1+8z+6z^2+4z^3-3z^4)}
      {(1+z)(1-z)(1-5z-z^2-3z^3)}\right), \\
\sphcs_{(H_4,X_{V_4})} &=& 
    \left(\frac{1+z}{1-z}\right)^4 +
    \crseries\left(\frac{2z(1+11z+18z^2+22z^3-3z^4-z^5)}
       {(1+z)(1-z)(1-6z-10z^3-z^4)}\right) \\
\sphcs_{(H_5,X_{V_5})} &=& 
    \left(\frac{1+z}{1-z}\right)^5 +
    \crseries\left(\frac{2z(1+14z+35z^2+20z^3+5z^4+6z^5-3z^6)}
       {(1+z)(1-z)(1-7z+2z^2-22z^3-3z^4-3z^5)}\right) 
\end{eqnarray*}
Computing further, for $1 \le n \le 9$, the growth series
\[
F_{\sphcpl(H_n,X_{V_n})} - \frac{1+(2n-1)z}{1-z} =
F_{\sphcpl \cap (X_{V_n}^*X_aX_{V_n}^*)}
\] 
has the form
\[
\frac{2z(1+(3n-1)z+ \cdots + ((-1)^{n-1}n+1)z^n + ((-1)^{n}-2)z^{n+1})}
           {(1+z)(1-z)(1-(n+2)z + \cdots + (((-1)^{n-1}-2)n+2)z^{n-1} + ((-1)^{n}-2)z^{n})};
\]
we have not yet found a pattern for the intermediate terms.

For comparison with the formula in Equation~(\ref{eq:necklaceformula})
from our earlier paper~\cite[Theorem~A]{chm} for the spherical
conjugacy growth series,
that result, along with the computation of the
spherical (conjugacy) growth series
of the free abelian group,
 yields
\begin{eqnarray*}
\sphcs_{(H_n,X_{V_n})} &=& 
  \frac{2z}{1-z} + \left(\frac{1+z}{1-z}\right)^n + 
  \neck\left(\left(\left(\frac{1+z}{1-z}\right)^{n}-1\right)\left(\frac{2z}{1-z}\right)\right),
\end{eqnarray*}
where $\neck$ is the operator in Equation~(\ref{eq:transformN}).

One route to compute the conjugacy geodesic growth series 
of the free product group $H_n$ as a rational function is to use the 
result of~\cite[Theorem~3.8]{CH14} that
$$
\geocs_{(H_n,X_{V_n})} = -1 + \geocs_{(\mathbb{Z},X_{a})}
  + \geocs_{(J_,X_{W_n})}
  - z \frac{d}{dz} \log [1 - (\geos_{(\mathbb{Z},X_{a})} - 1)(\geos_{(J_n,X_{W_n})} - 1)].
$$  
For the free abelian group $\mathbb{Z}$ on one generator,
 $\geocs_{(\mathbb{Z},X_{a})} = \geos_{(\mathbb{Z},X_{a})} =\frac{1+z}{1-z}$;
we compute the (conjugacy) geodesic growth series 
$\geocs_{(J_n,X_{W_n})} = 
 \geocs_{(J_n,X_{W_n})}$
for free abelian groups $\mathbb{Z}^n$ for all $n$
in Example~\ref{ex:freeabelian}.
\end{example}

In the next example, we use GAP computations
of conjugacy geodesic growth series for free abelian
groups on small numbers of generators along with a
result of Loeffler, Meier, and 
Worthington~\cite{LMW} in order to determine
a closed formula for the 
conjugacy geodesic growth series for all
free abelian groups.

\begin{example}\label{ex:freeabelian}
The free abelian group on $n$ generators 
$\mathbb{Z}^n = G_V = \langle a_1,...,a_n \mid [a_i,a_j]=1$ for all $i \neq j \rangle$
is the RAAG  with defining graph
$\Gamma_V$ that is the complete graph on the $n$ vertices in $V$.

For these abelian groups, each conjugacy class is
a singleton, and so growth and conjugacy growth coincide.
The spherical (conjugacy) growth series
of this group with respect to the usual generating set
$X_V = \{a_1,...,a_n\}^{\pm 1}$ can be computed completely as
$$
\sphcs_{(\mathbb{Z}^n,X_V)} = \sphs_{(\mathbb{Z}^n,X_V)} = \left(\frac{1+z}{1-z}\right)^n;
$$
that is, $\sphcs$ is the product of the spherical conjugacy
growth series of the $n$ vertex groups.

The geodesic and geodesic conjugacy growth series also
coincide for abelian groups.
In~\cite[Proposition~1]{LMW} Loeffler, Meier, and 
Worthington show that whenever 
$\geos_{(H,Y)} = \sum_{k=0}^\infty h_k z^k$ and
$\geos_{(J,Z)} = \sum_{k=0}^\infty j_k z^k$, then
the direct product of $H$ and $J$ has \wgeos
$\geos_{(H \times J,Y \cup Z)} = \sum_{k=0}^\infty m_k z^k$
where $m_k := \sum_{i=0}^k \binom{k}{i} h_i j_{k-i}$ for all $k$;
in~\cite[Theorem~3.8]{CH14} the first two authors
show that the same formula holds
for the conjugacy geodesic growth series of a direct product in terms of 
the conjugacy geodesic growth series of the direct factors.
However, in both cases this formula
does not directly show how to compute a rational expression for 
$\geocs(H \times J,Y \cup Z)$ from rational expressions for the conjugacy
growth series of the factors.
Using Equation~(\ref{eq:cycgeoregular})
and GAP again,
the conjugacy geodesic growth series of $\mathbb{Z}^n$ for
$2 \le n \le 5$ has the form
\begin{equation}\label{eq:freeabeliangeodesic}
\geocs_{(\mathbb{Z}^n,X_V)} = \geos_{(\mathbb{Z}^n,X_V)} = 
\frac{1+(n(3-n)/2)z + \cdots +(n!)z^n}{(1-z)(1-2z) \cdots (1-nz)} 
\end{equation}
(where we do not imply a pattern for the other coefficients
in the numerator here),
which led us to conjecture that the (conjugacy) geodesic 
growth series for $\mathbb{Z}^n$ for all $n$
is a rational function with denominator of the form
$(1-z)(1-2z) \cdots (1-nz)$, and hence the geodesic
growth function for $\mathbb{Z}^n$ satisfies
\[
\Theta_{\geol(\mathbb{Z}^n,X_V)}(k) = a_{n,1}(1^k)+a_{n,2}(2^k)+...+a_{n,n}(n^k)
\]
for all $k \ge 1$ and for some constants $a_{n,j}$.
The rational functions computed in GAP for $2 \le n \le 5$ 
satisfy $a_{n,j} = (-1)^{n-j} (2^j) \binom{n}{j}$.
Applying the formula above from~\cite{LMW,CH14}
for computing the (conjugacy) geodesic growth
series of a direct product
to $\mathbb{Z}^{n-1} \times \mathbb{Z}$, along with induction on $n$,
shows that this conjecture holds and that
this formula holds for the constants $a_{n,j}$ for all $n$, 
yielding Theorem~\ref{thm:freeabgeo}.
\end{example}

\begin{theorem}\label{thm:freeabgeo}
The geodesic and conjugacy geodesic growth series
for $\mathbb{Z}^n$ with respect to the standard generating set $X$
satisfy
\[ 
\geocs_{(\mathbb{Z}^n,X)} = \geos_{(\mathbb{Z}^n,X)} = 
1 + \sum_{j=1}^n (-1)^{n-j} (2^j) \binom{n}{j} \frac{jz}{1-jz}
\]
for all $n \ge 0$.
\end{theorem}

The above formula can also be obtained directly by counting geodesic paths of a given length $m$ starting at the origin in each of the $2^n$ orthants of $\mathbb{Z}^n$, (there are $n^m$ such geodesic paths in an orthant of $\mathbb{Z}^n$), noting that no geodesic can cross between orthants, and controlling the overcounting and undercounting in the intersections of orthants by inclusion-exclusion. 


\begin{example}\label{ex:3edgelinesegment2}
Returning to the RAAG $G=\langle a,b,c,d \mid [a,b]=[b,c]=[c,d]=1 \rangle$
from Example~\ref{ex:3edgelinesegment}, the algorithm in
KBMAG described above yields
expressions for
$F_{\sphcpl^{\{a,c\}}}$, $F_{\sphcpl^{\{a,c,d\}}}$,
and $F_{\sphcpl^{\{a,b,c,d\}}}$ as rational functions.
%
%
Plugging these into Equation~(\ref{eq:3edgesegment1})
and using the fact that $\crseries$ commutes with addition 
gives the formula
\begin{eqnarray*}\label{eq:3edgesegment2}
\sphcs_{(G_V,X_V)} &=& \frac{(1+z)(1+5z)}{(1-z)^2} +
\left(\frac{3(1+z)}{1-z}\right)\crseries\left(\frac{8z^2}{(1+z)(1-z)(1-3z)}\right)
\\ &&
+\crseries\left(\frac{8z^3(9-56z+31z^2)}{(1+z)(1-z)(1-3z)(1-5z)(1-4z-z^2)}\right).
\end{eqnarray*}
Again comparing this with the formula in Equation~(\ref{eq:necklaceformula}) from
our earlier paper~\cite[Theorem~A]{chm},
that result yields
\begin{eqnarray*}
\sphcs_{(G_V,X_V)} &=& \frac{(1+z)(1+5z)}{(1-z)^2}
+ \left(\frac{1+3z}{1-z}\right)\neck\left(\frac{4z^2}{(1-z)^2}\right)
\\ &&
+\neck\left(\frac{8z^2}{(1-z)(1-3z)}\right).
\end{eqnarray*}

The conjugacy geodesic growth series $\geocs$ for the RAAG $G$ 
is the rational function $\geocs = p(z)/q(z)$,
where
\begin{eqnarray*}
p(z) &=& 1-11z+41z^2-71z^3+47z^4+575z^5-2557z^6-189z^7+15796z^8 \\
&&       -21760z^9+5680z^{10}-6576z^{11}+6720z^{12}, \text{ and} \\
q(z) &=& (1+z)(1-z)(1-2z)(1-3z)(1-4z)(1-2z-z^2-2z^3) \\
&&       \cdot (1-8z+7z^2+24z^3-20z^4).
\end{eqnarray*}
\end{example}

We note that the GAP code described earlier in Section~\ref{sec:examples} 
and illustrated in
Example~\ref{ex:gapcodecycgeo} is not able to
compute the rational expression for $\geocs$ 
in Example~\ref{ex:3edgelinesegment2} directly;
as noted in the manual~\cite{kbmag} for the KBMAG package, 
when coefficients in the rational function are greater than
roughly 16,000, the software fails to complete the computation
of the growth series of the finite state automaton.
As a consequence, we computed the conjugacy geodesic growth
series with a slightly altered method, applying the principle
of inclusion-exclusion, which can be used 
more generally in the computation for $\geocs$ for any RAAG $G_V$ over an
Artin generating set $X_V$ as follows.
For each vertex $v \in V$, let $L_v'$ be the
language over $X_V$ defined by
\[
L_v' = \{w_1 x_v^e w_2 x_v^{-e} w_3 \mid e \in \{\pm 1\},
   w_1,w_3 \in X_{\Ne(v)}^*, \text{ and } w_2 \in X_V^*\}
\]
(where as before, $\Ne(v)$ is the set of all
vertices of $V$ that are adjacent to $v$ in the graph
$\Gamma=(V,E)$ defining $G_V$).
This language $L_v'$ is the set of all words $w$ over $X_V$
satisfying the property that for some cyclic permutation $w'$
of $w$, a finite sequence of shuffles can be applied to $w'$
to obtain a word with a subword of the form 
$x_vx_v^{-1}$ or $x_v^{-1}x_v$. 
Hence Propositions~\ref{prop:cycgeoisconjgeo}
and~\ref{prop:shuffle} show that words in $\geol(G_V,X_V)$
fail to be in $\geocl(G_V,X_V) = \geocpl(G_V,X_V)$ 
if and only if they lie on one of 
these languages $L_v'$; that is,
the language $\geocl(G_V,X_V)$ is 
obtained from $\geol(G_V,X_V)$ by
\[
\geocl(G_V,X_V) = 
\geol(G_V,X_V) - \bigcup_{v \in V} \left(L_v' \cap \geol(G_V,X_V)\right).
\]
Now the growth series for these languages satisfy
\begin{eqnarray}
\geocs_{(G_V,X_V)} &=& \geos_{(G_V,X_V)} 
- \sum_{v \in V} F_{\geol(G_V,X_V) \cap L_v'}
+ \sum_{\{v_1, v_2\} \subseteq V} F_{\geol(G_V,X_V) \cap L_{v_1}' \cap L_{v_2}'} \nonumber \\
&& + \cdots + (-1)^i\sum_{\{v_1, \ldots ,v_i\} \subseteq V} F_{\geol(G_V,X_V) \cap L_{v_1}' \cap \cdots \cap L_{v_i}'}
+ \cdots \nonumber \\
&& + (-1)^{|V|}F_{\geol(G_V,X_V) \cap L_{v_1}' \cap \cdots \cap L_{v_{|V|}}'}
 \label{eq:decompcycgeo}
\end{eqnarray}
where the sum in the $i$-th term 
$(-1)^i \sum_{\{v_1, \ldots, v_i\} \subseteq V} F_{\geol(G_V,X_V) \cap L_{v_1}' \cap \cdots \cap L_{v_i}'}$ 
is taken over the sets of $i$ distinct vertices $v_1,...,v_i$ of $V$.
The languages $L_v'$ are regular, and each $L_v'$ is the accepted
language of an FSA with 6 states which can be input into KBMAG.
The FSA for the language $\geol(G_V,X_V)$ can be computed as described
earlier in Section~\ref{sec:examples}, and commands in KBMAG/GAP can
be used to create FSAs and compute their growth
for the intersections of languages appearing in Equation~(\ref{eq:decompcycgeo}).


\section{Appendix}\label{sec:appendix}

In this appendix we give samples of GAP code
to do some of the computations discussed in Section~\ref{sec:examples}.

\begin{example}\label{ex:gapcodecycsl}
The GAP code below computes the growth series of the language
$\cycsl^{\{a,b\}}$ for the right-angled Artin group 
$G = \langle a,b \mid ~ \rangle$.

\bigskip

\begin{verbatim}
LoadPackage("KBMAG");
F:=FreeGroup("a","b"); a:=F.1; b:=F.2; 

#############################################
## Compute SL(G_V,X_V).                    ##

R:=KBMAGRewritingSystem(F);      AA:=AutomaticStructure(R);
W:=WordAcceptor(R);
## W is the FSA for SL(G_V,X_V) for the RAAG < a,b |  > 

#############################################
## Two functions used in computing CycSL   ##

Prefixes := function (fsa, i, n)
  local P,m,j,Holdp;
  Holdp:=fsa;    P := fsa;   m := NumberOfStatesFSA(Holdp);
  for j in [1..m] do    SetAcceptingFSA(P, j, false);   od;
  SetAcceptingFSA(P, i, true);
  return MinimizeFSA(P);
end;

Suffixes := function (fsa, i, n)
  local S,m,j,Holds;
  Holds:=fsa;    S := fsa;   m := NumberOfStatesFSA(Holds);
  for j in [1..m] do    SetInitialFSA(S, j, false);     od;
  SetInitialFSA(S, i, true);
  return MinimizeFSA(S);
end;

#############################################
## Compute CycSL(G_V,X_V)                  ##

n := NumberOfStatesFSA(NotFSA(W));

P := Prefixes(NotFSA(W), 1, n);      
S := Suffixes(NotFSA(W), 1, n);
CycPermNotSL := ConcatFSA(S,P);

for i in [2..n] do
  P := Prefixes(NotFSA(W), i, n);    
  S := Suffixes(NotFSA(W), i, n);
  Temp := OrFSA(CycPermNotSL, ConcatFSA(S, P));
  CycPermNotSL := Temp;
od;

CycSL := MinimizeFSA(NotFSA(CycPermNotSL));
## CycSL is the FSA for CycSL(G_V,X_V) for the RAAG < a,b |  >. 

##########################################################
## Build FSA for words with support U = {a,b}           ##

FSARequirea:=   rec( accepting := [ 2 ], 
  alphabet := rec( format := "dense", 
    names := [ _g1, _g2, _g3, _g4 ], printingFormat := "dense", 
    printingStrings := [ "_g1", "_g2", "_g3", "_g4" ], 
    size := 4, type := "identifiers" ), 
  denseDTable := [ [ 2, 2, 1, 1 ], [ 2, 2, 2, 2 ] ], 
  flags := [ "BFS", "DFA", "accessible", "minimized", "trim" ], 
  initial := [ 1 ], 
  isFSA := true, isInitializedFSA := true, 
  states := rec( size := 2, type := "simple" ), 
  table := rec( format := "dense deterministic", 
    numTransitions := 8, 
    printingFormat := "dense deterministic", 
    transitions := [ [ 2, 2, 1, 1 ], [ 2, 2, 2, 2 ] ] )   );

FSARequireb:=   rec( accepting := [ 2 ], 
  alphabet := rec( format := "dense", 
    names := [ _g1, _g2, _g3, _g4 ], printingFormat := "dense", 
    printingStrings := [ "_g1", "_g2", "_g3", "_g4" ], 
    size := 4, type := "identifiers" ), 
  denseDTable := [ [ 1, 1, 2, 2 ], [ 2, 2, 2, 2 ] ], 
  flags := [ "BFS", "DFA", "accessible", "minimized", "trim" ], 
  initial := [ 1 ], 
  isFSA := true, isInitializedFSA := true, 
  states := rec( size := 2, type := "simple" ), 
  table := rec( format := "dense deterministic", 
    numTransitions := 8, 
    printingFormat := "dense deterministic", 
    transitions := [ [ 1, 1, 2, 2 ], [ 2, 2, 2, 2 ] ] )   );

FSARequireab := AndFSA(FSARequirea,FSARequireb);
## FSARequireab accepts the words over X_V with support U = {a,b}.

#########################################################
## Compute CycSL^U and its growth series for U = {a,b} ##

CycSLab := AndFSA(CycSL, FSARequireab);
GrowthCycSLab := GrowthFSA(CycSLab);
## CycSLab is the FSA for CycSL^{a,b,c} for the RAAG < a,b |  >.
## GrowthCycSLab is the growth series for CycSL^{a,b,c}.
\end{verbatim}
\end{example}

\bigskip

\begin{example}\label{ex:gapcodecycgeo}
The following GAP code computes the 
conjugacy geodesic growth series 
for the right-angled Artin group 
$G = \langle a,b,c \mid [a,b] = [b,c] = 1 \rangle$.
This script also requires the
``LoadPackage("KBMAG");'' command and the two functions 
Prefixes and Suffixes in the GAP code in
Example~\ref{ex:gapcodecycsl}.

\bigskip

\begin{verbatim}

################################################
## Setting up the group and notation          ##

F:=FreeGroup("a","b","c"); a:=F.1; b:=F.2; c:=F.3;
G:=F/[a*b*a^-1*b^-1,b*c*b^-1*c^-1];


################################################
## Input the automata checking the languages  ##
##         L_v for all v in V = {a,b,c}       ##

_g1 := a; _g2 := a^-1; _g3 := b; _g4 := b^-1; _g5 := c; _g6 := c^-1;

FSAChecka := rec( accepting := [ 1 .. 4 ], 
  alphabet := rec( format := "dense", 
    names := [ _g1, _g2, _g3, _g4, _g5, _g6 ],
    printingFormat := "dense", 
    printingStrings := [ "_g1","_g2","_g3","_g4","_g5","_g6" ], 
    size := 6, type := "identifiers" ), 
  denseDTable := [ [ 2, 3, 1, 1, 4, 4 ], [ 2, 0, 2, 2, 4, 4 ], 
                   [ 0, 3, 3, 3, 4, 4 ], [ 2, 3, 4, 4, 4, 4 ] ], 
  flags := [ "BFS","DFA","accessible","minimized","trim" ], 
  initial := [ 1 ], 
  isFSA := true, isInitializedFSA := true, 
  states := rec( size := 4, type := "simple" ), 
  table := rec( format := "dense deterministic", 
    numTransitions := 22, 
    printingFormat := "dense deterministic", 
    transitions := [  [ 2, 3, 1, 1, 4, 4 ], 
                      [ 2, 0, 2, 2, 4, 4 ], 
                      [ 0, 3, 3, 3, 4, 4 ],
                      [ 2, 3, 4, 4, 4, 4 ]   ] )       );
##
FSACheckb := rec( accepting := [ 1 .. 3 ], 
  alphabet := rec( format := "dense", 
    names := [ _g1, _g2, _g3, _g4, _g5, _g6 ], 
    printingFormat := "dense", 
    printingStrings := [ "_g1","_g2","_g3","_g4","_g5","_g6" ], 
    size := 6, type := "identifiers" ), 
  denseDTable := [ [ 1, 1, 2, 3, 1, 1 ], [ 2, 2, 2, 0, 2, 2 ], 
                   [ 3, 3, 0, 3, 3, 3 ] ], 
  flags := [ "BFS","DFA","accessible","minimized","trim" ], 
  initial := [ 1 ], 
  isFSA := true, isInitializedFSA := true, 
  states := rec( size := 3, type := "simple" ), 
  table := rec( format := "dense deterministic", 
    numTransitions := 16, 
    printingFormat := "dense deterministic",
    transitions := [  [ 1, 1, 2, 3, 1, 1 ], 
                      [ 2, 2, 2, 0, 2, 2 ], 
                      [ 3, 3, 0, 3, 3, 3 ]   ] )      );
##
FSACheckc := rec( accepting := [ 1 .. 4 ], 
  alphabet := rec( format := "dense", 
    names := [ _g1, _g2, _g3, _g4, _g5, _g6 ],
    printingFormat := "dense", 
    printingStrings := [ "_g1","_g2","_g3","_g4","_g5","_g6" ], 
    size := 6, type := "identifiers" ), 
  denseDTable := [ [ 4, 4, 1, 1, 2, 3 ], [ 4, 4, 2, 2, 2, 0 ], 
                   [ 4, 4, 3, 3, 0, 3 ], [ 4, 4, 4, 4, 2, 3 ]], 
  flags := [ "BFS","DFA","accessible","minimized","trim" ], 
  initial := [ 1 ], 
  isFSA := true, isInitializedFSA := true, 
  states := rec( size := 4, type := "simple" ), 
  table := rec( format := "dense deterministic", 
    numTransitions := 22, 
    printingFormat := "dense deterministic", 
    transitions := [  [ 4, 4, 1, 1, 2, 3 ], 
                      [ 4, 4, 2, 2, 2, 0 ], 
                      [ 4, 4, 3, 3, 0, 3 ],
                      [ 4, 4, 4, 4, 2, 3 ]   ] )       );

##########################################################
## Build the FSA accepting the language Geo(G_V,X_V)    ##

Geo := MinimizeFSA(AndFSA(AndFSA(FSAChecka,FSACheckb),FSACheckc));

############################################################
## Build the FSA accepting the language CycGeo = ConjGeo  ##
##    and compute the growth seres of this language       ##

n := NumberOfStatesFSA(NotFSA(Geo));

P := Prefixes(NotFSA(Geo), 1, n);     S := Suffixes(NotFSA(Geo), 1, n);
CycPermNotGeo := ConcatFSA(S,P);

for i in [2..n] do
  P := Prefixes(NotFSA(Geo), i, n);     S := Suffixes(NotFSA(Geo), i, n);
  Temp := OrFSA(CycPermNotGeo, ConcatFSA(S, P));
  CycPermNotGeo := Temp;
od;

ConjGeo := MinimizeFSA(NotFSA(CycPermNotGeo));
GrowthConjGeo := GrowthFSA(ConjGeo);
## ConjGeo is the FSA for ConjGeo(G_V,X_V) for 
##   the RAAG < a,b,c | [a,b] = [b,c] = 1 >.
## GrowthConjGeo is the growth series for ConjGeo.
\end{verbatim}
\end{example}


\section*{Acknowledgements}


The first and third author were supported by the Swiss National Science 
Foundation grant Professorship FN PP00P2-144681/1. The first author was partially supported by the EPSRC Standard Grant EP/R035814/1. The second
author was  supported by a grant from
the Simons Foundation (Collaboration grant number 581433). The third author was also supported by the FCT Project UID/MAT
/00297/2019
(Centro de Matem\'{a}tica e Aplica\c{c}\~{o}es) and the FCT Project PTDC/MHC-FIL/2583/2014.
The authors thank Maranda Tomlinson for her input
on coding in GAP.


\bibliographystyle{plain}   
\bibliography{cgrefs_raag}


\textsc{Laura Ciobanu,
Mathematical and Computer Sciences,
Heriot-Watt University,      
Edinburgh EH14 4AS, UK}

\emph{E-mail address}{:\;\;}\texttt{l.ciobanu@hw.ac.uk}
\bigskip

\textsc{Susan Hermiller,
Department of Mathematics,
University of Nebraska,
Lincoln, NE 68588-0130, USA 
}

\emph{E-mail address}{:\;\;}\texttt{hermiller@unl.edu}

\bigskip

\textsc{V. Mercier,
Centro de Matem\'atica e Aplica\c{c}\~{o}es,
Faculdade de Ci\^{e}ncias e Tecnologia, Universidade Nova de Lisboa, 
2829-516 Caparica, Portugal
}

\emph{E-mail address}{:\;\;}\texttt{valen.mercier@gmail.com}
\medskip

\end{document}